\documentclass[10pt,a4paper]{article}
\usepackage[centertags]{amsmath}
\usepackage{amsfonts,amssymb,amsthm}
\usepackage{mathrsfs}

\newtheorem{proposition}{Proposition}[section]
\newtheorem{theorem}{Theorem}[section]

\newtheorem{lemma}{Lemma}[section]

\theoremstyle{definition}
\newtheorem{definition}{Definition}[section]

\theoremstyle{remark}
\newtheorem{example}{Example}

\newcommand{\R}{\mathbb{R}}

\newcommand{\N}{\mathbb{N}}

\newcommand{\M}{\mathbb{M}}

\newcommand{\MM}{\mathcal{M}}

\newcommand{\UU}{{\mathscr{U}}}

\newcommand{\nn}{{\mbox{\boldmath$n$}}}

\newcommand{\zz}{{\mbox{\boldmath$z$}}}

\newcommand{\vv}{{\mbox{\boldmath$v$}}}

\newcommand{\ww}{{\mbox{\boldmath$w$}}}

\newcommand{\tauV}{{\kern-3pt\tau}}

\newcommand{\nnu}{{\mbox{\boldmath$\nu$}}}

\newcommand{\Leb}[1]{{\mathscr L}^{#1}}      
\newcommand{\Wmod}{\mathcal W}

\renewcommand{\d}{{\mathrm d}}
\newcommand{\dx}{\d x}

\newcommand{\Rd}{{\R^d}}
\newcommand{\loc}{{\mathrm{loc}}}

\newcommand{\CE}{{\mathcal{CE}}}

\newcommand{\eps}{\varepsilon}
\newcommand{\dom}{\mathrm{dom}}
\newcommand{\Int}{\mathrm{int}}
\newcommand{\supp}{\mathrm{supp}}
\newcommand{\h}{\mathrm{h}}
\newcommand{\weakto}{\rightharpoonup}
\newcommand{\mom}{\tilde {\mathrm m}}
\newcommand{\restr}[1]{\lower3pt\hbox{$|_{#1}$}}

\newcommand{\la}{\langle}
\newcommand{\ra}{\rangle}

\newcommand{\AdmissibleDensity}{\mathcal A_p}
\newcommand{\AdmissibleDensitytwo}{\mathcal A_2}
\newcommand{\AdmissibleDensityab}{\mathcal A_p(a,b)}

\title{On a class of modified Wasserstein distances induced by concave mobility functions defined on bounded intervals}
\author{Stefano Lisini\footnote{Address: Dipartimento di Matematica ``F. Casorati'', Universit\`a degli Studi di Pavia, Via Ferrata 1, 27100 Pavia, Italy, e-mail: stefano.lisini@unipv.it},
Antonio Marigonda\footnote{Address: Dipartimento di Informatica, Universit\`a degli Studi di Verona, Strada Le Grazie 15, 37134 Verona, Italy, e-mail: antonio.marigonda@univr.it}}

\date{}

\begin{document}
\maketitle
\begin{abstract}
We study a new class of distances between Radon measures similar to those studied in \cite{DNS}.
These distances (more correctly pseudo-distances because can assume the value $+\infty$)
are defined generalizing the dynamical formulation of the Wasserstein distance
by means of a concave mobility function.
We are mainly interested in the physical interesting case (not considered in \cite{DNS}) of  a
concave mobility function defined in a bounded interval.
We state the basic properties of the space of measures endowed with this pseudo-distance.
Finally, we study in detail two cases: the set of measures defined in $\R^d$ with finite moments
and the set of measures defined in a bounded convex set.
In the two cases we give
sufficient conditions for the convergence of sequences with respect to the distance and
we prove a property of boundedness.
\end{abstract}



\section{Introduction}
In \cite{DNS}, Dolbeault, Nazaret and Savar\'e introduce and study the basic properties
of a new class of distances between non-negative Radon measures on $\R^d$.
These distances are defined generalizing the dynamical characterization of the Wasserstein distance.
We briefly recall that the Wasserstein distance between two non-negative measures with the same mass can be defined as a relaxed
optimal transportation problem (see \cite{Vil03}, \cite{AGS}, \cite{Vil09} for a reference on this interesting topic)
\begin{equation}\label{def:W}
W_p(\mu_0,\mu_1):=\inf\left\{
  \left(\int_{\Rd\times \Rd}|y-x|^p\,\d\Sigma\right)^{\frac1p}:\
  \Sigma\in\Gamma(\mu_0,\mu_1)
\right\}
\end{equation}
where $\Gamma(\mu_0,\mu_1)$ is the set of all \emph{transport plans} between $\mu_0$ and $\mu_1$:
they are non-negative measures $\Sigma$ on $\R^d\times \R^d$ with the same mass of $\mu_0$ and $\mu_1$ whose first
and second marginals are respectively $\mu_0$ and $\mu_1$, i.e. $\Sigma(B\times \R^d)=\mu_0(B)$
and $\Sigma(\R^d\times B)=\mu_1(B)$ for all Borel set $B$ of $\R^d$.

In  \cite{BB}, Benamou and Brenier prove that the Wasserstein distance defined in \eqref{def:W}
can be characterized, for absolutely continuous measures with respect to the Lebesgue measure $\Leb{d}$,
with compactly supported smooth densities, as follows
\begin{equation}\label{dynamicalformulationW}
  \begin{aligned}
    &W_p^p(\mu_0,\mu_1)= \inf\Big\{\int_0^1 \int_{\R^d}
    |\vv_t (x)|^p\rho_t(x)\,\d x\,\d t\,:\\
    &\quad
    \partial_t\rho_t+\nabla\cdot(\rho_t\mathbf\vv_t)=0\ \text{in
      $\R^d\times (0,1)$,}\quad \mu_0=\rho\restr{t=0}\Leb d,\quad
    \mu_1=\rho\restr{t=1}\Leb d\Big\}.
  \end{aligned}
\end{equation}
The proof of the dynamical characterization for general non-negative Borel measures was given in \cite{AGS}
where the continuity equation in \eqref{dynamicalformulationW} was considered in distributional sense.

The generalization of \eqref{dynamicalformulationW} studied in \cite{DNS},
roughly speaking, replace the mobility coefficient $\rho$ in \eqref{dynamicalformulationW}
with a non-linear one $h(\rho)$, where $h:[0,+\infty)\to [0,+\infty)$
is a concave increasing function such that $h(0)=0$
(particularly important examples are the functions $h(\rho)=\rho^\alpha$, $\alpha\ge0$)
and the new ``distance'' is defined modifying \eqref{dynamicalformulationW} as follows
\begin{equation}\label{def:Wh}
  \begin{aligned}
    & W_{p,h}^p(\mu_0,\mu_1)= \inf\Big\{\int_0^1 \int_{\R^d}
   |\vv_t(x)|^p h\big(\rho_t(x)\big)\,\d x\,\d t:\\
    &\quad
   \partial_t\rho_t+\nabla\cdot(h(\rho_t)\,\vv_t)=0\ \text{in
    $\Rd\times (0,1)$,}\quad
   \mu_0=\rho\restr{t=0}\Leb d,\quad
  \mu_1=\rho\restr{t=1}\Leb d\Big\}.
  \end{aligned}
\end{equation}
This ``definition'' is not rightly stated because it is necessary to specify
the spaces where $\rho$ and $\vv$ has to belong,
and the notion of solution of the modified continuity equation in \eqref{def:Wh}.
The right framework is that of Radon measures and distributional solutions.

The motivation for studying distances defined like in \eqref{def:Wh} arises from physical problems.
Indeed many interesting models are described by
partial differential equations whose solutions can be seen as trajectory
of the gradient flow of a suitable energy functional with respect to this distance
(see for instance the introduction of \cite{DNS} and \cite{CLSS}).

On the other hand, the concave mobility $h(\rho) \geq 0$
considered in \cite{DNS} is defined on the unbounded interval $[0,+\infty)$
and has to be necessarily non-decreasing.
If we want to consider non-monotone concave mobilities $h(\rho)\geq0$,
then the domain of $h$ has to be a bounded interval.
This case, not considered in \cite{DNS}, is physically interesting.
Indeed, examples of equations that can be modeled as gradient flows with respect to this kind of distances
are a version of Cahn-Hilliard equation \cite{EG},
some equation modelling chemotaxis with prevention of overcrowding \cite{BFD06,BDi,DiR},
equations describing the relaxation of gas of fermions
\cite{bib:k3,bib:k,bib:f1,bib:f2,CLR08,CRS08},
studies of phase segregation \cite{GL1, Sl08}, and studies of thin liquid
films \cite{Ber98}.

The principal example of mobility function in the papers cited above is
$$h(\rho)=\rho(1-\rho), \qquad\text{defined in } [0,1],$$  or $h(\rho)=1-\rho^2$ defined in $[-1,1]$,
mainly for the Cahn-Hilliard equation, the relaxation of fermion gas and the chemotaxis with overcrowding prevention.
A more general example is of the form
$h(\rho)=(\rho-a)^\alpha(b-\rho)^\beta$ defined in $[a,b]$ for some $\alpha,\beta\in[0,1]$.
In the previous examples, if $a<0$ then the density could be negative at some points and we have to consider
signed measures instead of non-negative measures.

In this paper we will show that almost all the properties of the distance studied in \cite{DNS}
can be extended to this case.

As previously observed, in order to give a precise meaning of the dynamical characterization \eqref{dynamicalformulationW}
and to define in a rigorous way the modified distance \eqref{def:Wh},
the right framework is that of time dependent families of Radon measures
and distributional solutions of the continuity equation.
Following the explanation given in the introduction of \cite{DNS}, we
replace $\rho_t$ by a continuous curve $t\in [0,1]\mapsto \mu_t$
($\mu_t=\rho_t\,\Leb d$ in the absolutely continuous case)
in the space $\MM^+(\R^d)$ of nonnegative Radon measures in $\R^d$
endowed with the usual weak$^*$ topology.
We replace the vector field $\vv_t$ in \eqref{dynamicalformulationW}
with a time dependent family of vector measures
$\nnu_t:=\vv_t\mu_t\ll\mu_t$.
The continuity equation in \eqref{dynamicalformulationW}
can be written in terms of the couple $(\mu,\nnu)$
\begin{equation}\label{eq:cap1:9}
  \partial_t \mu_t+\nabla\cdot\nnu_t=0\quad\text{in the sense of
    distributions in }\mathscr D'(\R^d\times(0,1)),
\end{equation}
and it is a linear equation.
Since $\vv_t=\d \nnu_t/\d \mu_t$
is the density of $\nnu_t$ with respect to $\mu_t$,
the action functional
which has to be minimized in \eqref{dynamicalformulationW} is
\begin{equation}\label{eq:cap1:10}
  \int_0^1\Phi(\mu_t,\nnu_t)\,\d t,\qquad
  \Phi(\mu,\nnu):=\int_\Rd \left|\frac{\d \nnu}{\d
      \mu}\right|^p\,
  \d\mu.
\end{equation}
In the case of  absolutely continuous measures with respect to $\Leb d$, i.e.
$\mu=\rho\Leb d$ and $\nnu=\ww\Leb d$,
the functional $\Phi$ can be expressed as
\begin{equation}\label{eq:cap1:12}
  \Phi(\mu,\nnu):=\int_{\R^d} \phi(\rho,\ww)\,\d\Leb
  d(x),\quad
  \phi(\rho,\ww):=\rho\left|\frac \ww\rho\right|^p.
\end{equation}
Denoting by $\CE(0,1)$ the class of measure-valued
distributional solutions $(\mu,\nnu)$ of the continuity equation
\eqref{eq:cap1:9}, we can state the dynamical
characterization of the Wasserstein distance as follows
\begin{equation}\label{eq:cap1:13}
  W_p^p(\mu_0,\mu_1)=\inf\Big\{\int_0^1\Phi(\mu_t,\nnu_t)\,\d t: (\mu,\nnu)\in \CE(0,1),\,\,
  \mu\restr{t=0}=\mu_0,\, \mu\restr{t=1}=\mu_1 \Big\}
\end{equation}
(as already observed, the Benamou-Brenier characterization \eqref{eq:cap1:13}
for Borel non-negative measures was proven in \cite{AGS}).
We observe that the function $\phi$ defined in \eqref{eq:cap1:12} is $p$-homogeneous w.r.t.\ $\ww$,
is convex with respect to $(\rho,\ww)$,
and positively $1$-homogeneous with respect to $(\rho,\ww)$.
By the $1$-homogeneity it is immediate to check that the functional $\Phi$ in \eqref{eq:cap1:12}
is independent on the Lebesgue measure, in the sense that if $\gamma \in \MM^+_\loc(\R^d)$ is another reference measure
such that $\supp (\gamma) = \R^d$ and $\mu=\tilde\rho\gamma$ and $\nnu=\tilde\ww\gamma$, then
\begin{equation}\label{eq:cap1:14}
    \Phi(\mu,\nnu)=\int_{\R^d} \phi(\tilde\rho,\tilde\ww)\,\d\gamma.
\end{equation}

We explain the main idea of \cite{DNS} for state rigorously the intuitive ``definition'' \eqref{def:Wh}.
Given a concave mobility function $h:(a,b)\to (0,+\infty)$, we consider
still the linear continuity equation \eqref{eq:cap1:9} and modify the action density $\phi$ in the following way:
$\phi:(a,b)\times \R^d\to[0,+\infty)$
\begin{equation}\label{eq:cap1:20}
  \phi(\rho,\ww):=h(\rho)\left|\frac\ww{h(\rho)}\right|^p.
\end{equation}
The concavity of $h$ is a necessary and sufficient
condition for the convexity of $\phi$ in \eqref{eq:cap1:20} (see
\cite{R} and Theorem \ref{pleg}).
We observe that $\phi$ still satisfies the $p$-homogeneity with respect to $\ww$ and is globally convex,
but it is no longer positively $1$-homogeneous with respect to $(\rho,\ww)$.
Hence, in order to consider the integral functional $\Phi$ like \eqref{eq:cap1:14}
it is necessary to precise the reference measure $\gamma\in\MM^+_\loc(\R^d)$ for that
$\rho$ and $\ww$ are the densities of $\mu$ and $\nnu$ respectively.
Defining $$\Phi(\mu,\nnu)=\int_{\R^d}\phi(\rho,\ww)\d\gamma$$ when $\mu=\rho\gamma$, $\nnu=\ww\gamma$,
and defining $\Phi$ suitably  on the singular part of $\mu$ and $\nnu$ with respect to $\gamma$,
(see Definition \ref{def:Phi})
the definition of the generalized Wasserstein distance associated to $(\phi,\gamma)$ is therefore
\begin{equation}\label{eq:cap1:16}
  \Wmod_{\phi,\gamma}^p(\mu_0,\mu_1):=\inf\Big\{\int_0^1\Phi(\mu_t,\nnu_t)\,\d t:
  (\mu,\nnu)\in \CE(0,1),\,\, \mu\restr{t=0}=\mu_0, \, \mu\restr{t=1}=\mu_1 \Big\}.
\end{equation}
Particularly important for the applications are the following choices of $\gamma$:
\begin{itemize}
\item $\gamma:=\Leb d_{|\Omega}=\chi_\Omega \Leb{d}$,
with $\Omega$ an open subset of $\R^d$;
\item $\gamma:=e^{-V}\Leb d$ for some $C^1$ potential $V:\R^d\to\R$;
\item $\gamma:=\mathscr H^k\restr{\M}$, where $\M$ is a smooth $k$-dimensional manifold embedded in $\R^d$
with the Riemannian metric induced by the Euclidean distance and $\mathscr H^k$ denotes the
$k$-dimensional Hausdorff measure.
\end{itemize}

In the paper \cite{CLSS}, the authors used this kind of distance in the case
$\gamma =\Leb d_{|\Omega}$ in order to study the problem of the convexity of
integral functionals along geodesics induced by the distance. The forthcoming
paper \cite{LMS} will be devoted to the study of forth orders equations
(Cahn-Hilliard type with nonlinear mobility and thin-film like equations), with
the proof of the existence of solutions by means of the minimizing movements
approximation scheme (see \cite{AGS}) for the distance like \eqref{eq:cap1:16}
and a first order integral functional.

We conclude this introduction stating the principal properties obtained in this paper for the distance like
\eqref{eq:cap1:16} with $h:(a,b)\to(0,+\infty)$,
referring to Section \ref{sec:W} for the precise definitions and the complete statements.
We recall that the choice of consider the mobility with bounded domain $(a,b)$
allow to consider also the distance between signed measures.

\begin{itemize}
\item
The space $\MM_\loc(\R^d)$ endowed with the distance
$\Wmod_{\phi,\gamma}$ is a complete pseudo-metric space (the
distance can assume the value $+\infty$), inducing as strong as,
or stronger topology than the weak$^*$ one.
Bounded sets with respect to $\Wmod_{\phi,\gamma}$ are weakly$^*$ relatively compact.
The distance $\Wmod_{\phi,\gamma}$ is lower semi continuous with respect to
the weak$^*$ convergence.
\item
In order to avoid that the distance could be $+\infty$ we consider
the space $\MM[\sigma]:= \big\{\mu\in \MM_\loc(\R^d): \Wmod_{\phi,\gamma}(\mu,\sigma)<+\infty\big \}$
for a given measure $\sigma \in \MM_\loc(\R^d)$.
The space $\MM[\sigma]$ turns out to be a complete metric space.
\item
$\MM[\sigma]$ is a geodesic space and the geodesic are unique if $h$ is strictly concave.
\item If $\mom_{-q}(\gamma)<+\infty$, where $q$ is the conjugate exponent of $p$
and the generalized momentum is defined in Definition \ref{def:momentum},
then $\mu(\R^d)=\sigma(\R^d)$ for every $\mu\in\MM[\sigma]$.
\end{itemize}
Finally, in Section \ref{connectivity} we give
sufficient conditions on the measures $\mu_0$, $\mu_1$ in order to have finiteness of the distance
$\Wmod_{\phi,\gamma}(\mu_0,\mu_1)$, and we prove two results: one for the all space $\R^d$ with the
Lebesgue measure as a reference, the other one for convex bounded domains in $\R^d$.
In the two cases we study also the relation between the weak-$*$ convergence of measures
and the convergence with respect to the distance $\Wmod_{\phi,\gamma}$.

\newpage

\section{Preliminaries}\label{pre}
In this Section we introduce the necessary tools
in order to define in the next Section the modified Wasserstein distance
and prove its basic properties.
The contents are an adaptation of Sections 2-4 of \cite{DNS}.
\subsection{Notation}\label{not}

Let $X$ be a topological space, $A\subset X$, $f:X\to \mathbb R\cup\{\pm \infty\}$ be a function.
We denote by:\\
\begin{tabular}{ll}&\\
$\mathrm{int}(A)$, $\bar A$, $\partial A$& the \emph{interior}, the \emph{closure} and the \emph{boundary} of $A$, respectively;\\ &\\
$\chi_A:X\to\{0,1\}$& the \emph{characteristic function} of $A$, namely\\ &$\chi_A(x)=1$ if $x\in A$, $\chi_A(x)=0$ if $x\notin A$;\\ &\\
$\mathrm{dom}(f):=\{x\in X:\,f(x)\in\mathbb R\}$&the \emph{(effective) domain} of $f$;\\ &\\
$\mathrm{epi}(f):=\{(x,\alpha)\in X\times\mathbb R:\,\alpha\ge f(x)\}$& the \emph{epigraph} of $f$;\\ &\\
$\mathrm{hypo}(f):=\{(x,\beta)\in X\times\mathbb R:\,\beta\le f(x)\}$& the \emph{hypograph} of $f$.\\ &\\
$\mathscr L^d$&the \emph{Lebesgue measure} on $\mathbb R^d$;\\ &\\
$\mathcal M_{\mathrm{loc}}(\mathbb R^d)$&the set of \emph{signed Radon measures} on $\mathbb R^d$;\\ &\\
$\mathcal M^+_{\mathrm{loc}}(\mathbb R^d)$&the set of \emph{non-negative Radon measures} on $\mathbb R^d$;\\ &\\
$\mathcal M_{\mathrm{loc}}(\mathbb R^d;\mathbb R^h)$&the set of $\mathbb R^h$-valued Radon measures on $\mathbb R^d$.\\ &\\
\end{tabular}\\
We say that $f$ is \emph{lower semicontinuous or l.s.c.} (resp. \emph{upper semicontinuous or u.s.c.}) iff $\mathrm{epi}(f)$ (resp. $\mathrm{hypo}(f)$) is closed in $X\times\mathbb R$.
If $(X,d)$ is a metric space, this is equivalent to say that $f$ is l.s.c. (resp. u.s.c.) iff $\displaystyle f(x)\le\liminf_{y\to x}f(y)$ (resp. $\displaystyle f(x)\ge\limsup_{y\to x}f(y)$).

\subsubsection{Push-forward of measures}
Given a Borel measure $\mu$ on a topological space X, and a Borel map $T:X \to Y$,
with values in a topological space $Y$,
we define the image measure of $\mu$ through the map $T$, denoted by $\nu=T_\# \mu$,
by  $\nu(B):=\mu(T^{-1}(B))$, for any Borel measurable set $B \subset Y$, or equivalently
\begin{equation}\label{def:pushforward}
    \int_Y \zeta(y) \,d\nu(y) = \int_X \zeta(T(x)) \, d\mu(x)\, , \quad \forall \zeta \in \mathcal C_b^0(Y).
\end{equation}
If $X$ and $Y$ are domains of $\R^d$, the map $T$ is sufficiently smooth and the measures $\mu$ and $\nu$ are
absolutely continuous with respect to Lebesgue measure with densities
$\widetilde \rho$ and $\rho$ respectively, then $\nu=T_\# \mu$ is
equivalent, by the change of variables theorem, to
\begin{equation}\label{dpf}
\rho(T(x))\,\det(DT(x))=\widetilde\rho(x).
\end{equation}
The formula \eqref{dpf} for the densities holds in a very greater generality (see \cite[Lemma 5.5.3]{AGS}).

\subsection{Convex Analysis}\label{cnv}
In this subsection we recall some concepts from convex analysis, our main reference is \cite{R}.

\begin{definition}[Recession functional]
Let $f:\mathbb R^N\to \mathbb R\cup\{+\infty\}$ be a proper convex function. The \emph{recession functional} $f^\infty$ of $f$
is the positively homogeneous proper convex function defined by
(cfr. \cite[Theorem 8.5, p.66]{R}):
$$f^\infty(y):=\sup\{f(x+y)-f(x):\,x\in\mathrm{dom}f\}.$$
If $f$ is l.s.c, then $f^\infty$ is l.s.c., and for any $x\in\mathrm{dom}(f)$ it holds:
$$f^\infty(y):=\lim_{\lambda\to+\infty}\frac{f(x+\lambda y)-f(x)}\lambda.$$
We have that:
\begin{enumerate}
\item if $0\in\mathrm{dom}(f)$, it holds
$\displaystyle f^\infty(y):=\lim_{\lambda\to+\infty}\frac{f(\lambda y)}\lambda$ for all $y\in\mathbb R^N$.
\item if $0\notin\mathrm{dom}(f)$, it holds
$\displaystyle f^\infty(y):=\lim_{\lambda\to+\infty}\frac{f(\lambda y)}\lambda$ for all $y\in\mathrm{dom}(f)$.
\end{enumerate}
\end{definition}

\begin{definition}[Concave-convex functions]
Let $C\subset\mathbb R^k$, $D\subset\mathbb R^d$ be convex sets, and $\tilde f:C\times D\to\mathbb R\cup\{\pm\infty\}$ be a function.
We will call $\tilde f$ a \emph{concave-convex} function if:
\begin{enumerate}
\item for each $z\in D$ the map $r\mapsto \tilde f(r,z)$ is concave,
\item for each $r\in C$ the map $z\mapsto \tilde f(r,z)$ is convex.
\end{enumerate}
Given a concave-convex function $\tilde f:C\times D\to\mathbb R$, we define its
\emph{lower extension} $\tilde f_1:\mathbb R^{k}\times\mathbb R^d\to\mathbb R\cup\{\pm\infty\}$
by setting:
$$\tilde f_1(r,z)=\left\{\begin{array}{lll}\tilde f(r,z)&\textrm{if}&r\in C,\,z\in D\\
+\infty&\textrm{if}&r\in C,\,z\notin D\\
-\infty&\textrm{if}&r\notin C \end{array}
\right.$$
$\tilde f_1$ is still a concave-convex function.
\end{definition}

\begin{theorem}[Partial Legendre]\label{pleg}
Let $f:\mathbb R^k\times\mathbb R^d\to\mathbb R\cup\{+\infty\}$ be convex and l.s.c.
Then the function defined by:
$$\tilde f(r,z):=\sup_{w\in\mathbb R^d}[\langle z,w\rangle- f(r,w)]$$
is a concave-convex function from $\mathbb R^k\times\mathbb R^d$ to $\mathbb R\cup\{\pm\infty\}$.
For every fixed $r$, the function $z\mapsto \tilde f(r,z)$ is l.s.c.
Conversely, given any concave-convex function $\tilde f:\mathbb R^k\times\mathbb R^d\to\mathbb R\cup\{\pm\infty\}$,
the function defined by:
$$f(r,w):=\sup_{z\in\mathbb R^d}[\langle z,w\rangle- \tilde f(r,z)]$$
is a convex map and for every fixed $r$, the function $w\mapsto f(r,w)$ is l.s.c.
Moreover, if $\mathrm{dom}(\tilde f)=C\times D$ and $\tilde f$ agrees with its lower extension,
then $f$ is l.s.c.
\end{theorem}
\begin{proof}
See \cite[Theorem 33.1]{R}.
\end{proof}

\subsection{Action function}

\begin{definition}[Admissible action density functions]
Let $\phi:\mathbb R\times\mathbb R^d\to[0,+\infty]$ be a l.s.c. nonnegative proper convex function, $1<p<\infty$.
We say that $\phi$ is an \emph{admissible action density of order} $p$ if it satisfies the following two properties:
\begin{enumerate}
\item[(F1)] $\ww\mapsto\phi(\cdot,\ww)$ is $p$-homogeneus,
 i.e. for every given $\rho\in\mathbb R$ such that $\{\rho\}\times\mathbb R^d\cap\mathrm{dom}(\phi)\ne \emptyset$
 we have $\phi(\rho,0)=0$ and for every $\lambda\ne 0$, $\ww\in\mathbb R^d$
 we have $\phi(\rho,\lambda \ww)=|\lambda|^p\phi(\rho, \ww)$ (both sides may be $+\infty$).
\item[(F2)] there exists $\rho_0\in \mathbb R$ such that $\{\rho_0\}\times\mathbb R^d\subseteq\mathrm{dom}(\phi)$
 and $\phi(\rho_0,\ww)>0$ for all $\ww\ne 0$.
\end{enumerate}
The set of all admissible action densities of order $p$ will be denoted by $\AdmissibleDensity$.
Given $a,b\in\R$, $a<b$ we will denote by $\AdmissibleDensityab$ the set of action densities in $\AdmissibleDensity$ such that
$\Int(\dom(\phi))=(a,b)\times\R^d$.\\
Let $q$ be the conjugate exponent of $p$. We construct the \emph{partial dual} $\mathcal A^*_q$ of $\mathcal A_p$
as follows. For all $\phi\in \mathcal A_p$, we define
the concave-convex function $\tilde\phi:\mathrm{dom}(\phi)\to\mathbb R\cup\{+\infty\}$ by setting:
\begin{equation}\label{def:phitilde}
\frac1q\tilde\phi(\rho,\zz):=\sup_{\ww\in \mathbb R^d}\Big\{\langle \zz,\ww\rangle-\frac1p\phi(\rho,\ww)\Big\}.
\end{equation}
We will call the lower extension of $\tilde\phi$ the \emph{marginal conjugate} of $\phi$ and we will still denote it by $\tilde\phi$.
We observe that $\tilde\phi$ is $q$-homogeneous with respect to the second variable and $\tilde\phi(\rho,\zz)\geq 0$.
We define:
$$\mathcal A^*_q:=\{\tilde \phi:\,\tilde\phi \textrm{ is the marginal conjugate of }\phi,\,\phi\in\mathcal A_p\}$$
and it is easy to check that $\Int(\dom(\tilde\phi))=(a,b)\times\R^d$ if $\phi\in\AdmissibleDensityab$.
\end{definition}

The following proposition can be proved exactly as Theorem 3.1 of \cite{DNS}.

\begin{proposition}[$\phi$-norm]\label{prop:phinorm}
Let $1<p<+\infty$, $q$ be the conjugate exponent of $p$ and $\phi\in\mathcal A_p$.
Then:
\begin{enumerate}
\item For every $\rho\in\mathbb R$ such that $\{\rho\}\times\mathbb R^d\subset\mathrm{dom}(\phi)$,
 the functions $\ww\mapsto\phi(\rho,\ww)^{1/p}$ and $\zz\mapsto\tilde \phi(\rho,\zz)^{1/q}$ are norms on $\mathbb R^d$ each one dual of the other.
We have:
\begin{equation}\label{def:phinorm}
\|\zz\|_{(\phi,\rho)*}:=\tilde\phi(\rho,\zz)^{1/q}=\sup_{\ww\ne 0}\frac{\langle \ww,\zz\rangle}{\phi(\rho,\ww)^{1/p}},\hspace{2cm}
\|\ww\|_{(\phi,\rho)}:=\phi(\rho,\ww)^{1/p}=\sup_{\zz\ne 0}\frac{\langle \ww,\zz\rangle}{\tilde \phi(\rho,\zz)^{1/q}}.
\end{equation}
\item The restriction to $\mathrm{dom}(\tilde \phi)$ of the marginal conjugate $\tilde\phi$ of $\phi$ takes its values in $[0,+\infty)$ and
it is a concave-convex function.
\item Given $\rho_0,\rho_1\in\mathbb R$ with $[\rho_0,\rho_1]\times\mathbb R^d\subseteq\mathrm{dom}(\phi)$, there exists a constant $C=C(\rho_0,\rho_1)$
such that for every $\rho\in[\rho_0,\rho_1]$ it holds:
$$C^{-1}|\ww|^p\le\phi(\rho,\ww)\le C|\ww|^p,\hspace{2cm}C^{-1}|\zz|^q\le\tilde\phi(\rho,\zz)\le C|\zz|^q,\hspace{1cm}\forall \ww,\zz\in\mathbb R^d.$$
\end{enumerate}
Equivalently, a function $\phi$ belongs to $\mathcal A_p$ if and only if it admits the dual representation formula
\begin{equation}\label{phileg}
\frac1p\phi(\rho,\ww):=\sup_{\zz\in \mathbb R^d}\Big\{\langle \zz,\ww\rangle-\frac1q\tilde\phi(\rho,\zz)\Big\},
\end{equation}
where $\tilde\phi:\mathbb R\times\mathbb R^d\to[0,\infty)$ is (the lower extension of) a nonnegative concave-convex function which is $q$-homogeneous with respect to $\zz$.
\end{proposition}

\begin{lemma}\label{homog}
Let $\phi\in\mathcal A_p$. Then the recession functional is $p$-homogeneous with respect to the second variable, i.e.
$\phi^\infty(\rho,\lambda \ww)=|\lambda|^p\phi^\infty(\rho,\ww)$ for $\lambda\in\mathbb R$. Moreover,
for $\rho\ne 0$ it is possible to express $\phi^\infty(\rho,\ww)=\rho\varphi^\infty(\ww/\rho)$,
where $\varphi^\infty:\mathbb R^d\to[0,+\infty]$ is  convex $p$-homogeneous
function such that $\varphi^\infty(\ww)> 0$ if $\ww\ne 0$.\\
\end{lemma}
\begin{proof}
We notice that $(0,0)$ may not belong in general to $\mathrm{dom}(\phi)$, however we have:
\begin{eqnarray*}
\phi^\infty(\rho,\ww)&:=&
\lim_{\lambda\to+\infty}\frac{\phi(\bar\rho+\lambda\rho,\lambda \ww)-\phi(\bar\rho,0)}\lambda
=\displaystyle\lim_{\lambda\to+\infty}\frac{\phi(\bar\rho+\lambda\rho,\lambda \ww)}\lambda
=\lim_{\lambda\to+\infty}\lambda^{p-1}\phi(\bar\rho+\lambda\rho,\ww),
\end{eqnarray*}
for every $\bar\rho\in\mathbb R$ such that $(\bar\rho,0)\in\mathrm{dom}(\phi)$, and
such $\bar\rho$ exists by definition of the class $\mathcal A_p$.
Hence $\phi^\infty$ is still $p$-homogeneous with respect to $\ww$.
The other statement follows from the definition of the class $\mathcal A_p$.
\end{proof}

We notice that in the case of $\phi\in\AdmissibleDensityab$ we have $\phi^\infty(0,0)=0$ and $\phi^\infty(\rho,\ww)=+\infty$ for $(\rho,\ww)\ne(0,0)$.

One of the most interesting example of admissible density function in $\AdmissibleDensityab$ is the following:
\begin{definition}\label{def:phih}
Let $p>1$ and $q$ its conjugate exponent.
Let $h:\mathbb R\to[0,+\infty)\cup\{-\infty\}$ be an u.s.c. concave function with $\Int(\dom(h))=(a,b)$, $a,b\in\R$, $a<b$,
$h(\rho)>0$ for every $\rho\in(a,b)$.
Define $\tilde\phi_h(\rho,\zz)=h(\rho)|\zz|^q$ on $\mathbb R\times\mathbb R^d$. We have that this is a concave-convex map
which is $q$-homogeneous with respect to $\zz$.
Hence, it is the marginal conjugate of the l.s.c. convex map
$\phi_h\in\AdmissibleDensityab$ defined by
\begin{equation}\label{scase}
\phi_h(\rho,\ww)=
\begin{cases}
\displaystyle\frac{|\ww|^p}{h(\rho)^{p-1}}&\textrm{if }\rho\in\mathrm{dom}(h),\,h(\rho)\ne 0\\ &\\
0&\textrm{if }h(\rho)=0,\,\ww=0\\ &\\
\displaystyle +\infty&\textrm{if }h(\rho)=0,\,\ww\ne0\textrm{ or } h(\rho)=-\infty.
\end{cases}
\end{equation}
Such function $h$ is called \emph{mobility function}.
\end{definition}

The following proposition shows that every admissible function $\phi$ is bounded from above by an admissible function of
the type \eqref{scase}.
\begin{proposition}\label{prop:phile}
If $\phi\in\AdmissibleDensityab$, then there exists a concave function $h$ such that $\Int(\dom(h))=(a,b)$, $h(r)>0$ for every $r\in(a,b)$ and
\begin{equation}\label{philphih}
    \phi(r,\ww)\leq\phi_h(r,\ww).
\end{equation}
\end{proposition}
\begin{proof}
Let us define
$$h(r):=\inf_{|\zz|=1}\tilde\phi(r,\zz),$$
where $\tilde\phi$ is defined in \eqref{def:phitilde}.
By the $q$-homogeneity of $\tilde\phi$ we have
$$\tilde\phi(r,\zz)\geq h(r)|\zz|^q.$$
Then, by the representation \eqref{phileg} for $\phi$ and $\phi_h$, we obtain
$$\frac{1}{p}\phi(r,\ww)\leq\sup_{\zz\in\R^d}\Big\{\la\ww,\zz\ra-\frac{1}{q}h(r)|\zz|^q \Big\}
    = \frac{1}{p}\phi_h(r,\ww).$$
\end{proof}

\subsection{Action functional}\label{mt}

Given an admissible action density function $\phi$ and a reference measure $\gamma$ on $\R^d$,
we can define the corresponding action functional.

\begin{definition}[$\phi$-Action functional]\label{def:Phi}
Let $\gamma\in\mathcal M^+_{\mathrm{loc}}(\mathbb R^d)$ be a reference measure and $\phi\in\mathcal A_p$.\\
For every $\mu\in\mathcal M_{\mathrm{loc}}(\mathbb R^d)$ and $\nnu\in\mathcal M_{\mathrm{loc}}(\mathbb R^d;\mathbb R^d)$
such that $\mathrm{supp}(\mu)$ and $\mathrm{supp}(\nnu)$ are contained in $\mathrm{supp}(\gamma)$ we can write
their Lebesgue decomposition $\mu=\rho\gamma+\mu^\perp$, $\nnu=\ww\gamma+\nnu^\perp$.
Introducing a nonnegative Radon measure
$\sigma\in\mathcal M^+_{\mathrm{loc}}(\mathbb R^d)$ such that
$\mu^\perp\ll\sigma$ and $\nnu^\perp\ll\sigma$
(e.g. take $\sigma=|\mu^\perp|+|\nnu^\perp|$)
and using the notation
$\mu^\perp=\rho^\perp\sigma$ and $\nnu^\perp=\ww^\perp\sigma$,
we define the action functional $\Phi$ associated to $\phi$ by
$$\Phi(\mu,\nnu|\gamma)=\Phi^a(\mu,\nnu|\gamma)+\Phi^\infty(\mu,\nnu|\gamma)
:=\int_{\mathbb R^d}\phi(\rho,\ww)\,d\gamma+\int_{\mathbb R^d}\phi^\infty(\rho^\perp,\ww^\perp)\,d\sigma.$$
Since $\phi^\infty$ is $1$-homogeneous, the definition does not depend on $\sigma$.
\end{definition}

We collect in the following theorem some properties of convex functionals on measures.
The proof can be found in \cite{DNS} (see also \cite{AFP} for functionals defined on measures).
\begin{theorem}[Properties of $\Phi$]\label{convfunc}
Let $\phi\in\mathcal A_p$ and $\Phi$ as in Definition \ref{def:Phi}.
\begin{enumerate}
\item \emph{Lower semicontinuity}.
If three sequences $(\gamma_n)_{n\in\mathbb N}\subset\mathcal M_{\mathrm{loc}}^+(\mathbb R^k)$,
$(\mu_n)_{n\in\mathbb N}\subset\mathcal M_{\mathrm{loc}}(\mathbb R^d)$,
$(\nnu_n)_{n\in\mathbb N}\subset\mathcal M_{\mathrm{loc}}(\mathbb R^d,\mathbb R^d)$
weakly* converge to $\gamma$, $\mu$, $\nnu$ respectively, then $\Phi(\mu,\nnu|\gamma)\le\displaystyle\liminf_{n\to+\infty}\Phi(\mu_n,\nnu_n|\gamma_n)$.
\item \emph{Monotonicity w.r.to $\gamma$}.
Assume that $(0,0)\in\dom(\phi)$
(in this case by homogeneity we have $\phi(0,0)=0$)
and let $\gamma_1,\gamma_2\in \mathcal M_{\mathrm{loc}}^+(\mathbb R^k)$ be
such that $\gamma_1\le\gamma_2$.
Then $\Phi(\mu,\nnu|\gamma_2)\le\Phi(\mu,\nnu|\gamma_1)$
for every $(\mu,\nnu)$ such that
$\mathrm{supp}(\mu) \cup \mathrm{supp}(\nnu)\subseteq\mathrm{supp}(\gamma_i)$, $i=1,2$.
\item \emph{Monotonicity w.r.to convolution}.
Let $k\in C^{\infty}_c(\mathbb R^d)$ be a convolution
kernel satisfying $k(x)\ge 0$ for all $x\in\mathbb R^d$ and $\int_{\mathbb R^d}k(x)\,dx=1$. Then
$\Phi(\mu\ast k, \nnu \ast k|\gamma\ast k)\le \Phi(\mu,\nnu|\gamma)$.
\end{enumerate}
\end{theorem}

The following example shows that the statement on
monotonicity with respect to the reference measure
may fail if $(0,0)\not\in \dom(\phi)$.
\begin{example}[Non-monotonicity]
Let $d=1$. We define $\phi:\mathbb R\times\mathbb R\to\mathbb R\cup\{+\infty\}$ to be
$\phi(r,\vv)=|\vv|^2$ if $r\in [3/2,2]$ and $+\infty$ elsewhere.
Define $\gamma_2=3/2\gamma_1=\chi_{[1,2]}(x)\mathscr L^1$ and set $\mu=\nnu=\gamma_2=3/2\gamma_1$.
Then
$$\Phi(\mu,\nnu|\gamma_2)=\int_{\mathbb R} \phi(1,1)\,d\gamma_2=+\infty.$$
$$\Phi(\mu,\nnu|\gamma_1)=\int_{\mathbb R} \phi(3/2,3/2)\,d\gamma_1=\frac{3}{2}.$$
Hence $\gamma_1<\gamma_2$ but $\Phi(\mu,\nnu|\gamma_1)<\Phi(\mu,\nnu|\gamma_2)$.
\end{example}

When  $\phi\in\AdmissibleDensityab$, the finiteness of the
corresponding action functional $\Phi(\mu,\nnu|\gamma)$, force
the absolute continuity of $\mu$ with respect to $\gamma$ and a
boundedness of the density of $\mu$ with respect to $\gamma$ .
We state this important property in the following proposition.

\begin{proposition}\label{assc1}
Let $\phi\in\AdmissibleDensityab$ and
$\gamma\in\mathcal M^+_{\mathrm{loc}}(\mathbb R^d)$ a fixed reference measure.
Let $\mu\in\mathcal M_{\mathrm{loc}}(\mathbb R^d)$, $\nnu\in\mathcal M_{\mathrm{loc}}(\mathbb R^d;\mathbb R^d)$ be such that $\Phi(\mu,\nnu|\gamma)<+\infty$.
Then $\mu \ll \gamma$, $\nnu \ll \gamma$ and
\begin{equation}\label{fi}
    \Phi(\mu,\nnu|\gamma)=\Phi^a(\mu,\nnu|\gamma)=\int_{\mathbb R^d}\phi(\rho,\ww)\,d\gamma,
\end{equation}
where $\mu=\rho\gamma$ and $\nnu=\ww\gamma$.
Moreover we have
\begin{equation}
    a\le \rho(x)\le b \qquad \text{for } \gamma \text{-a.e. } x\in\R^d.
\end{equation}
\end{proposition}

\begin{proof}
Since $\phi\in\AdmissibleDensityab$ and $p>1$, by the definition of $\phi^\infty$
and Lemma \ref{homog}, it is easy to check that
$\phi^\infty(\rho,\ww) = +\infty$ if $(\rho,\ww)\not=(0,0)$,
and $\phi^\infty(0,0)=0$.
If $\mu=\rho\gamma+\mu^\perp$, $\nnu=\ww\gamma+\nnu^\perp$ and
$\sigma\in\mathcal M^+_{\mathrm{loc}}(\mathbb R^d)$ such that $\mu^\perp=\rho^\perp\sigma$ and $\nnu^\perp=\ww^\perp\sigma$,
we can represent
$$\Phi^\infty(\mu,\nnu|\gamma)=\int_{\mathbb R^d}\phi^\infty(\rho^\perp,\ww^\perp)\,d\sigma.$$
In order to have $\Phi^\infty(\mu,\nnu|\gamma)<\infty$, we must have $\rho^\perp(x)=0$ and $\ww^\perp(x)=0$ for $\sigma$-a.e. $x\in\mathbb R^d$.
This implies that $\rho\ll\gamma$ and $\nnu\ll\gamma$ and \eqref{fi} holds.
The last statement follows from $\int_{\R^d}\phi(\rho,\ww)\,d\gamma <+\infty$.
\end{proof}

\subsection{Continuity equation}\label{ce}
In this Subsection we collect the basic facts on the measure solutions of the
continuity equation. It is an adaptation of \cite{DNS} and \cite{AGS}, with the
novelty that here we consider signed measures instead of non-negative measures.

\begin{definition}
Given $T>0$, we consider the \emph{continuity equation}:
\begin{equation}\label{cont1}
\partial_t\mu_t+\mathrm{div}(\nnu_t)=0,\hspace{2cm}\textrm{in }\mathbb R^d\times(0,T),
\end{equation}
where $\mu_t,\nnu_t$ are Borel families of measures in $\mathcal M_{\mathrm{loc}}(\R^d)$ and $\mathcal M_{\mathrm{loc}}(\mathbb R^d;\mathbb R^d)$
respectively, defined for $t\in(0,T)$ satisfying
\begin{equation}\label{fin1}
\int_0^T|\mu_t|(B(0,R))\,dt<+\infty,\hspace{2cm}V_R:=\int_0^T|\nnu_t|(B(0,R)\,dt<+\infty \hspace{1cm} \forall \,\, R>0,
\end{equation}
and the equation \eqref{cont1} holds in the sense of distributions, i.e.
\begin{equation}\label{cont2}
\int_0^T\int_{\mathbb R^d}\partial_t\zeta(x,t)\,d\mu_t(x)\,dt+\int_0^T\int_{\mathbb R^d}\nabla_x(\zeta(x,t))\,d\nnu_t(x)\,dt=0
,\hspace{1cm}\textrm{for every }\zeta\in C^1_c(\mathbb R^d\times(0,T)).
\end{equation}
We recall that, thanks to the disintegration theorem, we can identify $(\nnu_t)_{t\in[0,T]}$
with the measure $\nnu=\int_0^T\nnu_t\,dt\in\mathcal M_{\mathrm{loc}}(\mathbb R^d\times(0,T);\mathbb R^d)$ defined by:
$$\langle \nnu,\zeta\rangle=\int_0^T\left(\int_{\mathbb R^d}\zeta(x,t)d\nnu_t\right)\,dt,
\hspace{1cm}\forall\zeta\in C^0_c(\mathbb R^d\times(0,T);\mathbb R^d).$$
Similarly, we can identify $(\mu_t)_{t\in[0,T]}$ with a measure
$\mu=\int_0^T\mu_t\,dt\in\mathcal M_{\mathrm{loc}}(\mathbb R^d\times(0,T))$.
\end{definition}

\begin{lemma}
Let $T>0$ and $(\mu_t,\nnu_t)_{t\in(0,T)}$ be a Borel family of measures satisfying (\ref{fin1}) and (\ref{cont2}).
Then there exists a unique weakly* continuous curve
$[0,T]\ni t\mapsto\tilde\mu_t\in\mathcal M_{\mathrm{loc}}(\mathbb R^d)$
such that $\mu_t=\tilde\mu_t$ for $\mathcal L^1$-a.e. $t\in(0,T)$;
if $\zeta\in C^1_c(\mathbb R^d\times(0,T))$ and $t_1,t_2\in[0,T]$ with $t_1\le t_2$ we have:
$$\int_{\mathbb R^d}\zeta(t_2,x)\,d\mu_{t_2}-\int_{\mathbb R^d}\zeta(t_1,x)\,d\mu_{t_1}=
\int_{t_1}^{t_2}\left(\int_{\mathbb R^d}\partial_t\zeta(t,x)\,d\mu_t(x)+\int_{\mathbb R^d}\nabla_x(\zeta(t,x))\,d\nnu_t(x)\right)\,dt.$$
Moreover if $\tilde\mu_s(\mathbb R^d) \in \R$ for some $s\in[0,T]$ and $\displaystyle\lim_{R\to+\infty}R^{-1}V_R=0$, then the total mass
$\tilde\mu_t(\mathbb R^d)\in\R$ and is constant.
\end{lemma}

\begin{definition}[Solution of continuity equation]
Let $T>0$, we denote by $\mathcal {CE}(0,T)$ the set of time-dependent measures $(\mu_t,\nnu_t)_{t\in[0,T]}$ such that
\begin{enumerate}
\item $t\mapsto\mu_t$ is weakly* continuous in $\mathcal M_{\mathrm{loc}}(\R^d)$ satisfying \eqref{fin1};
\item $(\nnu_t)_{t\in [0,T]}$ is a Borel family satisfying \eqref{fin1};
\item $(\mu,\nnu)$ satisfies \eqref{cont2}.
\end{enumerate}
Given $\mu^1,\mu^2\in \mathcal M_{\mathrm{loc}}(\R^d)$, we denote the set of solutions connecting $\mu^1$ to $\mu^2$ (possibly empty) by
$\mathcal {CE}(0,T, \mu^1 \to \mu^2)=\{(\mu,\nnu)\in\mathcal {CE}(0,T):\mu_0=\mu^1,\mu_T=\mu^2 \}$.
Given $\gamma \in \mathcal M^+_{\mathrm{loc}}(\R^d)$ reference measure and $\phi \in \mathcal A_p$, we denote by
$\mathcal {CE}_{\phi,\gamma}(0,T; \mu^1 \to \mu^2)=\{(\mu,\nnu)\in\mathcal {CE}(0,T; \mu^1 \to \mu^2):\int_0^T\Phi(\mu_t,\nnu_t|\gamma)\,dt <+\infty \}$,
which is the set of solutions of the continuity equation
connecting $\mu^1$ to $\mu^2$ with finite energy.
We also use the notation
$\mathcal {CE}_{\phi,\gamma}(0,T):=\{(\mu,\nnu)\in\mathcal {CE}(0,T):\int_0^T\Phi(\mu_t,\nnu_t|\gamma)\,dt <+\infty \}$.
\end{definition}

\begin{lemma}\label{lemma:gluing}
The following properties hold:
\begin{enumerate}
\item (Time rescaling) Let $\tau:[0,T']\to [0,T]$ be a strictly increasing absolutely continuous map with absolutely continuous inverse $s=\tau^{-1}$.
Then $(\mu,\nnu)$ is a distributional solution of (\ref{cont2}) iff $(\hat\mu,\hat\nnu)$, where
$\hat\mu=\mu\circ\tau$ and $\hat\nnu=\tau'(\nnu\circ\tau)$ is a distributional solution of (\ref{cont2}) on $(0,T')$.
\item (Gluing solution) Let $(\mu^1,\nnu^1)\in\mathcal{CE}(0,T_1)$, $(\mu^2,\nnu^2)\in\mathcal{CE}(0,T_2)$ with $\mu^1_{T_1}=\mu^2_0$.
Then the new family $(\mu_t,\nnu_t)_{t\in (0,T_1+T_2)}$, defined by $(\mu_t,\nnu_t)=(\mu^1_t,\nnu^1_t)$ for $0\le t\le T_1$
and $(\mu_t,\nnu_t)=(\mu^2_{t-T_1},\nnu^2_{t-T_1})$ for $T_1\le t\le T_2$, belongs to $\mathcal {CE}(0,T_1+T_2)$.
\end{enumerate}
\end{lemma}

\subsubsection{Conservation of the mass for solutions with finite energy}
In this paragraph we prove that, under a condition on the generalized moments of the reference measure $\gamma$
and for $\phi\in \AdmissibleDensityab$, the total (signed) mass conserves for solutions of the continuity equation
with finite energy.

\begin{definition}[Upper uniform concave bound]\label{def:uucb}
Let $\phi\in \AdmissibleDensityab$.
Fixing $\bar\rho:=(a+b)/2$ we use the notation
\begin{equation}\label{normbar}
    \|\ww\|:=\|\ww\|_{(\phi,\bar\rho)}, \qquad \|\zz\|_*:=\|\zz\|_{(\phi,\bar\rho)*},
\end{equation}
where the norms above (equivalents to the euclidean one) are defined in \eqref{def:phinorm}.
We consider the set:
$$\mathcal H:=\{g:\mathbb R\to\mathbb R\cup\{-\infty\}\,:g\textrm{ is u.s.c. and concave},\,g(\rho)\ge \tilde\phi(\rho,\zz/\|\zz\|_*)\,\,\forall \zz\ne 0\}.$$
This set is nonempty, and we can define:
$$\h(\rho)=\inf\{g(\rho):\,g\in\mathcal H\},$$
which turns out to be the smallest u.s.c. concave function greater than or equal to $\sup\{\tilde\phi(\rho,\zz):\,\|\zz\|_*=1\}$.
Since $\Int(\dom(\h))=(a,b)$ we obtain that
\begin{equation}\label{hbound}
 \h_{max}:=\sup_{\rho\in\R} \h(\rho)<+\infty.
\end{equation}
By homogeneity property it is immediate to prove that
\begin{equation}\label{eqvnorm}
    \tilde \phi(\rho,\zz)\le\mathrm{h}(\rho)\|\zz\|^q_* \qquad \text{and} \qquad \|\ww\|\le \mathrm{h}(\rho)^{1/q}\phi(\rho,\ww)^{1/p}.
\end{equation}
\end{definition}
When $\phi$ is given as in Definition \ref{def:phih}, we have
$\mathrm{h}(\rho)=C\cdot h(\rho)$, where $C:=\max\{|\zz|/\|\zz\|_*:\,\zz\ne 0\}$,
and $|\cdot|$ is the euclidean norm.

\begin{definition}\label{def:momentum}
Let $\gamma\in\mathcal M^+_{\mathrm{loc}}(\mathbb R^d)$, $r\in\mathbb R$. We define the
\emph{generalized $r$-th momentum} $\tilde {\mathrm m}_r(\gamma)$ of $\gamma$ by setting:
$$\mom_r(\gamma):=\gamma(B(0,1))+\int_{\mathbb R^d\setminus B(0,1)}|x|^r\,d\gamma(x).$$
\end{definition}
We observe that if $\mom_r(\gamma)<+\infty$ then $\mom_s(\gamma)<+\infty$ for every $s\leq r$.

\begin{proposition}[Mass conservation]\label{prop:consmass}
Let $p>1$, $q$ its conjugate exponent and $\phi\in \AdmissibleDensityab$.
Let $r\in\R$ such that $r\ge -q$ and $\gamma\in\mathcal M_{\mathrm{loc}}^+(\mathbb R^d)$
be a reference measure satisfying $\mom_r(\gamma)<+\infty$.\\
If $(\mu_t,\nnu_t)_{t\in[0,T]}\in\mathcal{CE}_{\phi,\gamma}(0,T)$ and
$\mu_0(\mathbb R^d) \in \R$, then $\mu_t(\mathbb R^d) = \mu_0(\mathbb R^d)$ for every $t\in [0,T]$.
\end{proposition}
\begin{proof}
We consider a cutoff function $\zeta\in C^\infty_c(\R^d)$ such that
$\zeta(x)=1$ if $|x|\leq1$, $\zeta(x)=0$ if $|x|\geq2$ and
$|\nabla \zeta (x)|\leq 1$ for all $x\in\R^d$.
We consider the family $\zeta_R(x)=\zeta(x/R)$, for $R>0$,
that obviously satisfies $|\nabla \zeta_R(x)|\leq 1/R$ for all $x\in\R^d$.

Using the notations of Definition \ref{def:uucb},
for every $t_1,t_2\in[0,T]$, $t_1<t_2$,
by Proposition \ref{prop:phinorm} and \eqref{eqvnorm} we have
\begin{align*}
   \Big| \int_{\R^d} \zeta_R\,d\mu_{t_1} -\int_{\R^d}  \zeta_R\,d\mu_{t_2}\Big|
   & \leq \int_{t_1}^{t_2}\int_{\R^d}| \nabla \zeta_R \cdot \ww_t |\,d\gamma\,dt \\
   & \leq \int_{t_1}^{t_2}\int_{\R^d} \tilde\phi(\rho_t,\nabla\zeta_R)^{1/q}\phi(\rho_t,\ww_t)^{1/p}\,d\gamma\,dt \\
    &\leq  \Big( \int_{t_1}^{t_2}\int_{B_{2R}\setminus B_R} \h(\rho_t) \|\nabla\zeta_R\|^q_*\,d\gamma\,dt\Big)^{1/q}
    \Big(\int_{t_1}^{t_2}\int_{\R^d} \phi(\rho_t,\ww_t)\,d\gamma\,dt \Big)^{1/p}.
\end{align*}
Since $\int_0^T\Phi(\mu_t,\nnu_t|\gamma)\,dt <+\infty$, by \eqref{hbound}
and the equivalence of $\| \cdot\|_*$ with the euclidean norm,
the last inequality shows that there exists $C>0$ such that
\begin{equation}\label{a}
\Big| \int_{\R^d} \zeta_R\,d\mu_{t_1} -\int_{\R^d} \zeta_R\,d\mu_{t_2}\Big| \leq C \Big(\frac{1}{R^q}\gamma(B_{2R}\setminus B_R)\Big)^{1/q}.
\end{equation}
Since $\mom_r(\gamma)<+\infty$ shows that $\lim_{R\to+\infty}R^r\gamma(B_{2R}\setminus B_R)=0$ we have that
$\lim_{R\to+\infty}\frac{1}{R^q}\gamma(B_{2R}\setminus B_R)=0$ if $r\geq -q$. Then the conservation of the mean follows from \eqref{a}.
\end{proof}

\begin{example}
When $\phi\in\AdmissibleDensityab$ with $a<0$ and $b>0$, if $\gamma = \Leb{d}$ and $d>q$ in general
solutions of the continuity equations with finite energy could not conserve the mass.

Let $\eps>0$ such that $a+\eps<0$ and $b-\eps>0$ and consider an initial measure with
  compact support and mass different from $0$,
  $\mu_0=\rho_0\Leb d$, such that $a+\eps\leq\rho_0\leq b-\eps$.
We define the curve, for $t\ge 0$,
  \begin{equation}\label{eq:cap5:13}
    \mu_t:=\rho_t\Leb d,\quad
    \rho_t(x):=e^{-dt }\rho_0(e^{-t}x),\quad
    \nnu_t:=\ww_t\Leb d = x\rho_t(x)\Leb d.
  \end{equation}
  It is easy to check that
  $(\mu,\nnu)\in \CE(0,{+\infty})$, $\mu_t(\Rd)=\mu_0(\R^d)$ and $a+\eps\leq\rho_t\leq b-\eps$.
By 3 of Proposition \ref{prop:phinorm}
we have that $\phi(\rho_t,\ww_t)\leq C|\ww_t|^p$.
By a simple computation we obtain that
\begin{align*}
    \int_{\R^d} |\ww_t(x)|^p\,\d x=
    \int_{\R^d} |x|^p e^{-tdp}|\rho_0(e^{-t}x)|^p\,\d x=
    \int_{\R^d} |y|^p e^{t((1-d)p+d)}|\rho_0(y)|^p\,\d y
\end{align*}
and then
$$
\int_0^{+\infty} \phi(\rho_t,\ww_t) \d x \,\d t <+\infty
$$
when $d>q$.

The curve $(\mu_t,\nnu_t)$ can be reparametrized between $[0,1]$
setting $s=\frac{2}{\pi}\arctan{t}$, $t\in (0,+\infty)$ and $\eta_s=\rho_{\tan (\frac{\pi}{2}s)}=\rho_t$.
It is not difficult to check that
the energy is still finite and $\eta_s$ connect $\mu_0$ with the null measure.
\end{example}

\subsubsection{Compactness for solutions with finite energy}
In this section we prove a compactness result for signed solutions of the continuity equation.
This result is a useful tool in order to obtain existence of geodesics of the distance defined in the next Section and
its lower semi-continuity with respect to weak$^*$ convergence.

\begin{proposition}[Compactness]\label{prop:comp}
Let $\phi\in \AdmissibleDensityab$
and $\gamma^n,\gamma\in\mathcal M^+_{\mathrm{loc}}(\mathbb R^d)$ be a sequence such that $\gamma^n\rightharpoonup^*\gamma$.\\
If $(\mu^n,\nnu^n)$ is a sequence in $\mathcal{CE}_{\phi,\gamma^n}(0,T)$ satisfying
\begin{equation}\label{bound}
    \sup_{n\in\N}\int_0^T\Phi(\mu^n_t,\nnu^n_t|\gamma^n)\,dt <+\infty ,
\end{equation}
then there exists a subsequence (still indexed by $n$) and a couple $(\mu,\nnu)\in\mathcal {CE}_{\phi,\gamma}(0,T)$
such that $\mu^n_t\rightharpoonup^*\mu_t$ in $\mathcal{M}_{\mathrm{loc}}(\mathbb R^d)$ for all $t\in[0,T]$,
$\nnu^n\rightharpoonup^*\nnu$ in $\mathcal M_{\mathrm{loc}}(\mathbb R^d\times(0,T);\mathbb R^d)$, and
\begin{equation}\label{lscn}
\int_0^T\Phi(\mu_t,\nnu_t|\gamma)\,dt\le\liminf_{n\to +\infty}\int_0^T\Phi(\mu_t^n,\nnu_t^n|\gamma^n)\,dt.
\end{equation}
If along the subsequence $\mom_{-q}(\gamma^n)<+\infty$ for all $n$ and $\mom_{-q}(\gamma)<+\infty$, and
$\mu_0^n(\R^d)\to\mu_0(\R^d)\in\R$, then $\mu_t^n(\R^d)\to\mu_t(\R^d)$ for every $t\in[0,T]$.
\end{proposition}

\begin{proof}
By Proposition \ref{assc1} we have
$\mu^n=\rho^n\gamma^n$, $\nnu^n=\ww^n\gamma^n$ and $|\rho^n| \leq c:=\max\{|a|,|b|\}$.
Then there exists a subsequence (still indexed by $n$) and $\rho$ such that
$\rho^n \weakto \rho$ weakly in $L^1_{\loc}(\R^d\times[0,T])$.
On the other hand, by \eqref{eqvnorm}, for every bounded Borel set $B\subset \R^d$ and for every $t_1,t_2\in[0,T]$, $t_1<t_2$ we have
\begin{align*}
    \int_{t_1}^{t_2}\int_B \|\ww^n\|\,d\gamma^n\,dt & \leq \int_{t_1}^{t_2}\int_B \h(\rho^n)^{1/q}\phi(\rho^n,\ww^n)^{1/p}\,d\gamma^n\,dt \\
    &\leq \Big( \int_{t_1}^{t_2}\int_B \h(\rho^n)\,d\gamma^n\,dt\Big)^{1/q}
    \Big(\int_{t_1}^{t_2}\int_B \phi(\rho^n,\ww^n)\,d\gamma^n\,dt \Big)^{1/p}.
\end{align*}
By \eqref{hbound}, \eqref{bound} and the equivalence of $\| \cdot\|$ with the euclidean norm, the last inequality shows that there exist $C>0$ such that
$$\int_{t_1}^{t_2}\int_B \|\ww^n\|\,d\gamma^n\,dt \leq C ((t_2-t_1) \gamma^n(B))^{1/q},$$
By this estimate there exist $\nnu\in\MM_\loc(\R^d\times[0,T],\R^d)$ and a subsequence
such that $\nnu^n \weakto^*\nnu$.
By the lower semicontinuity property of Theorem \ref{convfunc} we obtain \eqref{lscn}.
Reasoning as in the proof of Lemma 4.5 of \cite{DNS} we obtain that $(\mu,\nnu)$ satisfies the continuity equation.\\
Finally, by Proposition \ref{prop:consmass} $\mu_t^n(\R^d)$ and $\mu_t(\R^d)$ do not depend on $t\in[0,T]$.
\end{proof}

\section{The modified Wasserstein distance}\label{sec:W}

In this Section we give the rigorous definition of the modified Wasserstein distance illustrated in the introduction.
We deal only with the case of the distance induced by an action density function $\phi \in \AdmissibleDensityab$
for $a,b\in\R$ and we refer to \cite{DNS} for the case $\phi \in \AdmissibleDensity (0,+\infty)$.

The proofs are almost all omitted because follows exactly as in \cite[Section 5]{DNS} from the results of the previous Sections.

\begin{definition}
Given a reference measure $\gamma\in\MM^+_\loc(\R^d)$,
an admissible action density function $\phi \in \AdmissibleDensityab$ and
the corresponding action functional $\Phi$ of Definition \ref{def:Phi},
for $\mu^0,\mu^1 \in \MM_\loc(\R^d)$ we define
\begin{align}\label{def:Wm}
\Wmod_{\phi,\gamma}(\mu^0,\mu^1):=&\inf\left\{ \Big(\int_0^1
\Phi (\mu_s,\nnu_s|\gamma) \,\d s\Big)^{1/p} :
(\mu,\nnu)\in \CE_{\phi,\gamma}(0,1;\mu^0 \to\mu^1)\right\}.
\end{align}
\end{definition}

$\Wmod_{\phi,\gamma}(\mu^0,\mu^1)=+\infty$ if the set of
connecting curves $\CE_{\phi,\gamma}(0,1;\mu^0\to\mu^1)$ is empty.

By the compactness  Proposition \ref{prop:comp} we obtain the existence of constant speed minimizing geodesics.
Precisely, following the proof of \cite[Thm. 5.4]{DNS} and Theorem 5.11 of \cite{DNS} we can prove the following result.

\begin{proposition}[Existence of geodesics, convexity and uniqueness of geodesics]\label{prop:convexity}
Given $\gamma\in\MM^+_\loc(\R^d)$ and $\phi \in \AdmissibleDensityab$,
for every $\mu_0,\ \mu_1\in\MM_\loc(\R^d)$ such
that $\Wmod_{\phi,\gamma}(\mu_0,\mu_1)<+\infty$ there exists a
minimizing couple $(\mu,\nnu)$ in \eqref{def:Wm}
and the curve $(\mu_s)_{s\in[0,1]}$ is a constant speed geodesic for
$\Wmod_{\phi,\gamma}$, thus satisfying
\begin{equation*}
\Wmod_{\phi,\gamma}(\mu_t,\mu_s) =
|t-s|\Wmod_{\phi,\gamma}(\mu_0,\mu_1) \qquad \forall s,t \in
[0,1].
\end{equation*}
We have the characterization
\begin{align}\label{charWm}
\Wmod_{\phi,\gamma}(\mu^0,\mu^1)= \inf\left\{\int_0^1 \Big(\Phi
(\mu_s,\nnu_s|\gamma)\Big)^{1/p} \,\d s : (\mu,\nnu)\in \CE(0,1;\mu^0
\to\mu^1)\right\}.
\end{align}
Moreover $\Wmod^p_{\phi,\gamma}:\MM_\loc(\R^d)\times \MM_\loc(\R^d)\to[0,+\infty]$ is convex,
i.e. for every $\mu_i^j\in \MM_\loc(\Rd)$, $i,j=0,1$,
and $\tau\in [0,1]$, if $\mu^\tau_i=(1-\tau)\mu^0_i+\tau\mu^1_i$,
\begin{equation}\label{eq:40}
    \Wmod^p_{\phi,\gamma}(\mu^\tau_0,\mu^\tau_1)\le
    (1-\tau)\Wmod^p_{\phi,\gamma}(\mu^0_0,\mu^0_1)+\tau
    \Wmod^p_{\phi,\gamma}(\mu^1_0,\mu^1_1).
\end{equation}
If $\phi$ is strictly convex then
for every $\mu_0,\mu_1\in \MM_\loc(\Rd)$ with
$\Wmod_{\phi,\gamma}(\mu_0,\mu_1)<+\infty$ there exists
a \emph{unique} minimizer $(\mu,\nnu)\in \CE_{\phi,\gamma}(0,1;\mu_0\to\mu_1)$
of \eqref{def:Wm}.
\end{proposition}

\begin{proposition}\label{prop:basic}
Given $\gamma\in\MM^+_\loc(\R^d)$ and $\phi \in \AdmissibleDensityab$,
we have that $\Wmod_{\phi,\gamma}$ is a pseudo-distance on $\MM_\loc(\R^d)$;
i.e. $\Wmod_{\phi,\gamma}$ satisfies the axiom of the distance but
can assume the value $+\infty$.\\
The topology induced by $\Wmod_{\phi,\gamma}$ on $\MM_\loc(\R^d)$
is stronger than or equivalent to the weak$^*$ one.\\
Bounded sets with respect to $\Wmod_{\phi,\gamma}$ are weakly$^*$ relatively compact.
\end{proposition}
\begin{proof}
The verification of the axioms of the distance is straightforward except for the triangular inequality where we use
the gluing of solutions of Lemma \ref{lemma:gluing} and the characterization \eqref{charWm}.

In order to prove the topological property, reasoning as in the proof of Proposition \ref{prop:consmass} we obtain that
\begin{equation}
   \Big| \int_{\R^d} \zeta\,d\mu_1 -\int_{\R^d}  \zeta\,d\mu_0\Big|
   \leq  \sup|\nabla\zeta|(\h_{max}\gamma(\supp(\zeta)))^{1/q}\Wmod_{\phi,\gamma}(\mu_0,\mu_1)
\end{equation}
for every $\zeta\in C^1_c(\R^d)$.
Since $C^1_c(\R^d)$ is dense in $C_c(\R^d)$ we obtain the assertion on the topology induced by the distance
and on the relative compactness.
\end{proof}

The following lower semi-continuity result can be proved exactly as Theorem 5.6 of \cite{DNS} by using the compactness
Proposition \ref{prop:comp}.

\begin{proposition}[Lower semi-continuity]\label{prop:lsc}
\label{prop:semi} If $\gamma^n\weakto^*\gamma$ in $\MM^+_\loc(\R^d)$,
$\mu_0^n\weakto^*\mu_0$, $\mu_1^n\weakto^*\mu_1$ in $\MM_\loc(\R^d)$
and $\phi^n,\phi\in\AdmissibleDensityab$, such that $\phi^n\leq\phi^{n+1}$ and $\phi^n$ converges pointwise to $\phi$, then
\begin{equation}\label{lsc}
\liminf_{n\to +\infty}\Wmod_{\phi^n,\gamma^n}(\mu_0^n,\mu_1^n)\geq
\Wmod_{\phi,\gamma}(\mu_0,\mu_1) .
\end{equation}
\end{proposition}

The following completeness result can be proved as in Theorem 5.7 of \cite{DNS} ad using Proposition \ref{prop:lsc}.
The final assertion about the equality of the signed mass follows from Proposition \ref{prop:consmass}.
\begin{proposition}[Completeness and equality of the mass]\label{prop:completness}
Given $\gamma\in\MM^+_\loc(\R^d)$ and $\phi \in \AdmissibleDensityab$,
we have that the space  $\MM_\loc(\R^d)$ endowed with the pseudo-distance $\Wmod_{\phi,\gamma}$
is complete.\\
Given a measure $\sigma \in \MM_\loc(\R^d)$, the space
$\MM[\sigma]:= \big\{\mu\in \MM_\loc(\R^d):
\Wmod_{\phi,\gamma}(\mu,\sigma)<+\infty\big \}$ is a complete metric space.\\
If $\mom_{-q}(\gamma)<+\infty$ then $\mu(\R^d)=\sigma(\R^d)$ for every $\mu\in\MM[\sigma]$.
\end{proposition}

The following results follows from 3 and 4 of Theorem \ref{convfunc}.

\begin{proposition}[Monotonicity]\label{prop:monotonicity}
If $\phi_1\leq\phi_2$ then
\begin{equation*}\label{mon1}
\Wmod_{\phi_1,\gamma}(\mu_0,\mu_1)\leq \Wmod_{\phi_2,\gamma}(\mu_0,\mu_1),
\end{equation*}
for every $\mu_0$, $\mu_1\in\MM_\loc(\R^d)$.\\
Moreover, if $(0,0)\in\dom(\phi_i)$, $i=1,2$ and $\gamma_1\leq\gamma_2$ then
\begin{equation*}\label{mon2}
\Wmod_{\phi_1,\gamma_2}(\mu_0,\mu_1)\leq \Wmod_{\phi_2,\gamma_1}(\mu_0,\mu_1),
\end{equation*}
for every $\mu_0$, $\mu_1\in\MM_\loc(\R^d)$.
\end{proposition}

\begin{proposition}[Approximation by convolution]\label{prop:convolution}
Let $k\in C^\infty_c(\R^d)$ be a
nonnegative convolution kernel, with $\int_{\R^d}k(x)\,\dx =1$ and
$\supp(k)=\overline B_1(0)$, and let
$k_\eps(x):=\eps^{-d}k(x/\eps)$.
For every $\mu_0,\mu_1\in\MM(\R^d)$
\begin{align}
    \Wmod_{\phi,\gamma\ast k_\eps}(\mu_0\ast k_\eps,\mu_1\ast k_\eps)&\le
    \Wmod_{\phi,\gamma}(\mu_0,\mu_1)\quad
    \forall\, \eps>0;\\
    \lim_{\eps\to 0} \Wmod_{\phi,\gamma\ast k_\eps}(\mu_0\ast k_\eps,\mu_1\ast k_\eps)&=
    \Wmod_{\phi,\gamma}(\mu_0,\mu_1).
\end{align}
\end{proposition}

The following proposition deals with a control of the moments and a comparison between
the convergence with respect to  $\Wmod_{\phi,\gamma}$
and the standard Wasserstein distance defined in \eqref{def:W}.

\begin{proposition}
Let $\gamma\in\MM^+_\loc(\R^d)$ be satisfying $\mom_r(\gamma)<+\infty$ for some $r\in\R$
and $\phi \in \AdmissibleDensityab$.\\
If $\mu_0,\ \mu_1\in\MM_\loc(\R^d)$ satisfy $\Wmod_{\phi,\gamma}(\mu_0,\mu_1)<+\infty$,
then, setting $C:=\max\{|a|,|b|\}$, we have
\begin{equation}\label{momest}
    \mom_\delta (|\mu_i|) \leq C \mom_r(\gamma), \qquad \text{for }i=0,1, \quad \forall\,\delta\leq r.
\end{equation}
If $r\geq 1$ and $a\geq 0$ then the convergence with respect to  $\Wmod_{\phi,\gamma}$ in $\MM[\sigma]$,
for some non-negative measure $\sigma$ satisfying $\sigma(\R^d)<+\infty$,
implies the convergence with respect to the $r$-Wasserstein distance $W_r$.
\end{proposition}
\begin{proof}
Denoting by $1\vee |x|=\max\{1,|x|\}$, given $C=\max\{|a|,|b|\}$, $\delta\leq r$ and a Borel set $A\subset \R^d$,
by Proposition \ref{assc1} we obtain
\begin{equation}\label{b}
    \int_A (1\vee |x|)^\delta \,\d|\mu_i|(x) = \int_A (1\vee |x|)^\delta |\rho_i(x)|\,\d\gamma(x)
    \leq C \int_A (1\vee |x|)^\delta \,\d\gamma(x) \leq  C \int_A (1\vee |x|)^r \,\d\gamma(x).
\end{equation}
Choosing $A=\R^d$ in \eqref{b} we obtain \eqref{momest}.\\
If $\mu_n $ is a sequence in $\MM[\sigma]$ converging to $\mu$ with respect to $\Wmod_{\phi,\gamma}$,
then, by Proposition \ref{prop:basic}, $\mu_n$ weakly$^*$ converges to $\mu$ and,
by Proposition \ref{prop:completness}, $\mu_n(\R^d)=\mu(\R^d)=\mu(\sigma)$ because of the assumption on the moment of $\gamma$ and $r\geq 1$.
By \eqref{b} with $\delta=0$ we have that the sequence $\mu_n$ is tight and then $\mu_n$ narrowly converges to $\mu$.
Since \eqref{b} implies that the $r$-moments of $\mu_n$ are uniformly equiintegrable we obtain that
(see Lemma 5.1.7 of \cite{AGS}) $\mom_r(\mu_n)$ converges to $\mom_r(\mu)$ and we conclude.
\end{proof}

In particular the previous Proposition applies to the case $\gamma(\R^d)<+\infty$.

In the next proposition we state a simple comparison with the standard Wasserstein distance \eqref{def:W}.
\begin{proposition}[Comparison with Wasserstein distance]\label{prop:comparison}
Let $p>1$, $\phi\in\AdmissibleDensity (0,M)$, $\Omega\subset\R^d$ an open convex set
and $\gamma_\Omega=\chi_\Omega\Leb{d}$.
If $\mu_i$, $i=0,1$, are two absolutely continuous measures with respect to $\gamma_\Omega$,
$\mu_i=\rho_i\gamma_\Omega$, such that $0\leq \rho_i(x) \leq M'<M$,
$\mom_p(\mu_i)<+\infty$ for $i=0,1$  and $\mu_0(\R^d)=\mu_1(\R^d)$,
then there exists a constant $C$, depending only on $M'$, $\phi$ and $p$, such that
\begin{equation}\label{comparison}
    \Wmod_{\phi,\gamma_\Omega}(\mu_0,\mu_1) \leq C W_p(\mu_0,\mu_1) < +\infty,
\end{equation}
where $W_p$ denotes the standard $p$-Wasserstein distance.
\end{proposition}
\begin{proof}
Let $h$ be given by Proposition \ref{prop:phile}.
Since $h$ is concave and positive on $(0,M)$, we have that
$$ h(\rho)\geq \frac{h(M')}{M'}\rho, \qquad \forall\,\rho\in(0,M'),$$
and, consequently,
\begin{equation}\label{compar}
    \phi(\rho,\ww)\leq \frac{|\ww|^p}{h(\rho)^{p-1}}\leq \left(\frac{M'}{h(M')}\right)^{p-1}\frac{|\ww|^p}{\rho^{p-1}}
    \qquad \forall\rho\in(0,M').
\end{equation}
Since the $p$-moments of $\mu_0$ and $\mu_1$ are finite,
taking the geodesic interpolant $\mu_t$ between $\mu_0$ and $\mu_1$
for the standard $p$-Wasserstein distance $W_p(\mu_0,\mu_1)$, and denoting by $\rho_t$ the density of $\mu_t$,
we have that $\rho_t \leq M'$ (see the proof of \cite[Theorem 5.24]{DNS}). Since $\Omega$ is convex,
the support of $\mu_t$ belongs to $\overline\Omega$ and, denoting by $\psi(\rho,\ww):= \frac{|\ww|^p}{\rho^{p-1}}$,
we have that $\Wmod_{\psi,\gamma_\Omega}=W_p$ for all the measures with support in $\overline\Omega$.
Then, by \eqref{compar}, and recalling Proposition \ref{prop:monotonicity} we obtain \eqref{comparison}.
\end{proof}

\section{Measures at finite distance and convergence}\label{connectivity}
In this section we give sufficient conditions for the finiteness of the distance between two measures.
We study also the relation between the convergence with respect to the distance and the weak-$*$ one.
The first result concerns measures defined on the whole space $\R^d$ with the reference measure $\gamma=\Leb{d}$,
whereas the second one deals with measures defined on a bounded convex domain $\Omega$
with the reference measure $\gamma_\Omega=\Leb{d}_{|\Omega}$.

\subsection{The case of reference measure  $\Leb{d}$}

\begin{theorem}[Connectivity in $\R^d$]\label{th:connectivityrd}
Let $p>1$ and $\phi\in\AdmissibleDensity (0,M)$.
If $\mu_i$, $i=0,1$, are two absolutely continuous measures $\mu_i=\rho_i\Leb{d}$, such that $0\leq \rho_i(x) \leq M$,
$\mom_p(\mu_i)<+\infty$ for $i=0,1$  and $\mu_0(\R^d)=\mu_1(\R^d)$,
then there exists a constant $C>0$ depending only on $\phi$, $d$ and $p$ such that
\begin{equation}\label{eq:df}
    \Wmod_{\phi,\Leb{d}}(\mu_0,\mu_1) \leq C(\mom_p(\mu_0)+\mom_p(\mu_1)) < +\infty.
\end{equation}
\end{theorem}
We observe that the inequality \eqref{eq:df} holds in the case of the standard Wasserstein distance
(it is a very easy consequence of the definition \eqref{def:W}).
\begin{proof}
Let $h$ be given by Proposition \ref{prop:phile}.
Since $h$ is concave and non-negative, there exists $\tilde h:[0,M] \to [0,+\infty)$
of the form $\tilde h(\rho)= A\rho(M/B -B\rho)$ for $A,B>0$
such that $\tilde h (\rho) \leq h(\rho)$ in $[0,M]$.
Hence $\Wmod_{\phi,\Leb{d}}(\mu_0,\mu_1)\leq\Wmod_{\phi_h,\Leb{d}}(\mu_0,\mu_1)
\leq \Wmod_{\phi_{\tilde h},\Leb{d}}(\mu_0,\mu_1)$.
Thanks to this observation, it is sufficient to prove the result under the assumption that
$M=1$, $h(\rho)=\rho(1-\rho)$ and $0\leq \rho_i(x)\leq 1$, for $i=0,1$.

Defining
\begin{equation}\label{def:tildemu}
    \tilde \mu_i=\tilde \rho_i \Leb{d} = 2\mathrm{Id}_\#\mu_i,
\end{equation}
where $\mathrm{Id}$ denotes the identity map in $\R^d$,
we prove that there exists a constant $C_{p,d}$, depending only on $p$ and $d$, such that
\begin{equation}\label{eq:df2}
    W_{\phi_h,\Leb{d}}(\mu_i,\tilde\mu_i)<C_{p,d}\mom_p(\mu_i) \quad \text{ for } i=0,1.
\end{equation}
Indeed, for $t\in[0,1]$, taking $T_t(x) := (1+t^p)x$ and $\mu_t:=(T_t)_\#\mu_i=\rho_t\Leb{d}$,
by \eqref{dpf} we have that $\rho_t(y)=\frac{1}{(1+t^p)^d}\,\rho_i\left(\frac{y}{1+t^p}\right)$.
Defining $\vv_t(x):=\dot T_t\circ T^{-1}_t(x)=\frac{(pt^{p-1})}{1+t^p}\,x,$
and $\ww_t=\vv_t\rho_t$, $\nnu_t=\ww_t\Leb{d}$ it is easy to check that
$(\mu_t,\nnu_t)_{t\in(0,1)}\in \CE(0,1;\mu_i \to \tilde\mu_i)$.
By elementary computations, using the definition of $\mu_t$ and $\vv_t$, we have
\begin{align*}
\int_0^1\int_{\R^d}\frac{|\ww_t(x)|^p}{(\rho_t(x)(1-\rho_t(x)))^{p-1}}\,dx\,dt
&= \int_0^1\int_{\R^d}\frac{|\vv_t(x)|^p\rho_t(x)}{(1-\rho_t(x))^{p-1}}\,dx\,dt \\
&= \int_0^1\int_{\R^d}\frac{|\vv_t(x)|^p}{(1-\rho_t(x))^{p-1}}\,d\mu_t(x)\,dt \\
&= \int_0^1\int_{\R^d}\frac{|\vv_t(T_t(x))|^p}{(1-\rho_t(T_t(x)))^{p-1}}\,d\mu_i(x)\,dt\\
&= \int_0^1\int_{\R^d}\frac{(pt^{p-1})^p|x|^p}{(1-(1+t^p)^{-d}\rho_i(x))^{p-1}}\,d\mu_i(x)\,dt.
\end{align*}
Since $\rho_i(x)\leq 1$ and $(1+t^p)^{d} \geq 1 +dt^p$ we have
\begin{equation*}
    \frac{1}{1-(1+t^p)^{-d}\rho_i(x)} \leq \frac{1}{1-(1+t^p)^{-d}} = \frac{(1+t^p)^{d}}{(1+t^p)^{d}-1} \leq \frac{(1+t^p)^{d}}{dt^p}.
\end{equation*}
Then
\begin{align*}
    \int_0^1\int_{\R^d}\frac{(pt^{p-1})^p|x|^p}{(1-(1+t^p)^{-d}\rho_i(x))^{p-1}}\,d\mu_i(x)\,dt
    &\leq \int_0^1\int_{\R^d}\frac{(pt^{p-1})^p(1+t^p)^{d(p-1)}}{(dt^p)^{p-1}}|x|^p\,d\mu_i(x)\,dt \\
    & \leq \mom_p(\mu_i) \int_0^1 p^p d^{1-p}(1+t^p)^{d(p-1)} \,dt,
\end{align*}
and \eqref{eq:df2} follows with $\displaystyle{C_{p,d}=\int_0^1 p^p d^{1-p}(1+t^p)^{d(p-1)} \,dt}$.

Finally, by the triangular inequality, we have
\begin{equation}\label{triang}
\Wmod_{\phi_h,\Leb{d}}(\mu_0,\mu_1) \leq  \Wmod_{\phi_h,\Leb{d}}(\mu_0,\tilde\mu_0)+ \Wmod_{\phi_h,\Leb{d}}(\tilde\mu_0,\tilde\mu_1)+ \Wmod_{\phi_h,\Leb{d}}(\tilde\mu_1,\mu_1).
\end{equation}
Since by \eqref{dpf} we have $\tilde \rho_i(x)= 2^{-d}\rho_i(x/2)\leq 2^{-d}$ and
$\mom_p(\tilde\mu_i)=2^p \mom_p(\mu_i)<+\infty$, by Proposition
\ref{prop:comparison} applied to $\tilde\mu_0,\tilde\mu_1$, and observing that
$W_p(\tilde\mu_0,\tilde\mu_1)\leq \mom_p(\tilde\mu_0)+\mom_p(\tilde\mu_1)$ (it
is a simple consequence of the definition \eqref{def:W}), by \eqref{eq:df2} and
\eqref{triang} we obtain \eqref{eq:df} .
\end{proof}

Given $M>0$ and $c>0$ we define the set of measures
$$\MM^+_{M,c}(\R^d):=\{\mu\in\MM^+(\R^d): \mu=\rho\Leb{d},\, 0\leq\rho\leq M, \, \mu(\R^d)=c,\, \mom_p(\mu)<+\infty\}.$$

\begin{theorem}\label{th:equivconv}
Let $p>1$ and $\phi\in\AdmissibleDensity (0,M)$.
If $(\mu_n)_{n\in\N}$ is a sequence in $\MM^+_{M,c}(\R^d)$ weakly-$*$ convergent to $\mu\in\MM^+_{M,c}(\R^d)$, such that
\begin{equation}\label{convmom}
\mom_p(\mu_n) \to \mom_p(\mu),
\end{equation}
then
$$\lim_{n\to+\infty}\Wmod_{\phi,\Leb{d}}(\mu_n,\mu)=0.$$
\end{theorem}
\begin{proof}
Let $\bar\mu=\bar\rho\Leb{d}\in\MM^+_{M,c}$ be a fixed auxiliary measure such that
$M':=\sup\bar\rho <M$.\\
For every $\lambda\in (0,1)$, we define the convex combinations
$\mu_n^\lambda := (1-\lambda)\mu_n+\lambda\bar\mu$
and $\mu^\lambda := (1-\lambda)\mu+\lambda\bar\mu$.
Denoting by $\rho_n^\lambda$ the density of $\mu_n^\lambda$ with respect to $\Leb{d}$ we have
$\rho_n^\lambda\leq 1-\lambda(M-M')$.
By Proposition \ref{prop:comparison} and the convexity of the $p$-power of the standard $p$-Wasserstein distance
(Proposition \ref{prop:convexity} applied to $\phi(\rho,\ww)=|\ww|^p/\rho^{p-1}$ or \cite{Vil03}) we have
\begin{equation}\label{ooo}
\Wmod_{\phi,\Leb{d}}^p(\mu_n^\lambda,\mu^\lambda) \leq C W_p^p(\mu_n^\lambda,\mu^\lambda)
\leq C(1-\lambda)W_p^p(\mu_n,\mu).
\end{equation}
By the convergence of the $p$-moments \eqref{convmom}
and the weak-$*$ convergence we have (see \cite{AGS} or \cite{Vil03})
\begin{equation}\label{o}
    \lim_{n\to+\infty}W_p(\mu_n,\mu)=0.
\end{equation}
Moreover for the convexity of $\Wmod_{\phi,\Leb{d}}^p$ (Proposition \ref{prop:convexity}) we have
\begin{equation}\label{oo}
    \Wmod_{\phi,\Leb{d}}^p(\mu_n,\mu_n^\lambda)\leq\lambda\Wmod_{\phi,\Leb{d}}^p(\mu_n,\bar\mu),\qquad
    \Wmod_{\phi,\Leb{d}}^p(\mu,\mu^\lambda)\leq\lambda\Wmod_{\phi,\Leb{d}}^p(\mu,\bar\mu).
\end{equation}
Since
\begin{equation}
     \Wmod_{\phi,\Leb{d}}(\mu_n,\mu) \leq  \Wmod_{\phi,\Leb{d}}(\mu_n,\mu_n^\lambda)+
     \Wmod_{\phi,\Leb{d}}(\mu_n^\lambda,\mu^\lambda)+ \Wmod_{\phi,\Leb{d}}(\mu^\lambda,\mu),
\end{equation}
by \eqref{ooo}, \eqref{o} and \eqref{oo}
we have
\begin{equation}\label{ab}
 \limsup_{n\to +\infty} \Wmod_{\phi,\Leb{d}}(\mu_n,\mu)
 \leq \lambda^{1/p} \Big(\sup_{n}\Wmod_{\phi,\Leb{d}}(\mu_n,\bar\mu) + \Wmod_{\phi,\Leb{d}}(\mu,\bar\mu)\Big).
\end{equation}
By \eqref{convmom} and Theorem \ref{th:connectivityrd} we obtain
\begin{equation}\label{boundwmod}
    \sup_n\Wmod_{\phi,\Leb{d}}(\mu_n,\bar\mu) < +\infty.
\end{equation}
Since $\lambda>0$ is arbitrary, \eqref{ab} and \eqref{boundwmod} imply
$$\limsup_{n\to +\infty} \Wmod_{\phi,\Leb{d}}(\mu_n,\mu)=0$$
and we conclude.
\end{proof}

We recall that the convergence with respect to the standard Wasserstein distance $W_p$ is equivalent to the
weak-$*$ convergence and the convergence of the $p$-moments $\mom_p$ (see \cite{Vil03} or \cite{AGS}).
Consequently, Theorem \ref{th:equivconv} states that the convergence with respect to $W_p$ in $\MM^+_{M,c}(\R^d)$
implies the convergence with respect to $\Wmod_{\phi,\Leb{d}}$ for every $\phi\in\AdmissibleDensity (0,M)$.
We observe that this property is not true in the case of $\phi\in\AdmissibleDensity (0,+\infty)$,
where only a result like Proposition \ref{prop:comparison} hold (see Theorem 5.24 of \cite{DNS}).

\subsection{The case of the reference measure $\chi_{\Omega} \Leb{d}$ with $\Omega$ bounded convex}
When the reference measure is $\gamma_\Omega:=\chi_{\Omega} \Leb{d}$,
where $\Omega$ is a bounded convex smooth  domain, we have the following
result of finiteness of the distance and of boundedness of the space of admissible measures.
\begin{theorem}\label{th:finOmega}
Let $\phi\in\AdmissibleDensitytwo (a,b)$
and $\gamma_\Omega:=\chi_{\Omega}\Leb{d}$ with  $\Omega\subset \R^d$ a bounded convex smooth domain.
For every $c\in(a\Leb{d}(\Omega),b\Leb{d}(\Omega))$ we define the set of measures
$$\MM_{(a,b),c}(\Omega) := \{\mu\in\MM(\overline\Omega):
\mu=\rho\gamma_{\Omega},\,\, a\leq\rho\leq b,\,\, \mu(\overline\Omega)=c \}.$$
The space $\MM_{(a,b),c}(\Omega)$ endowed with the distance $\Wmod_{\phi,\gamma_\Omega}$ is bounded.
In particular $\Wmod_{\phi,\gamma_\Omega}(\mu_0,\mu_1)<+\infty$ for every $\mu_0,\mu_1\in \MM_{(a,b),c}(\Omega)$.
\end{theorem}

\begin{proof}
Defining $\mu_\infty:=\frac{c}{\Leb{d}(\Omega)}\gamma_\Omega$, we prove that
\begin{equation}
    \sup_{\mu_0\in \MM_{(a,b),c}(\Omega)} \Wmod_{\phi,\gamma_\Omega}(\mu_0,\mu_\infty)<+\infty.
\end{equation}

Let $h$ be given by Proposition \ref{prop:phile}.

For $\mu_0=\rho_0\gamma_\Omega \in \MM_{(a,b),c}(\Omega)$,
let $\rho:(0,+\infty)\times \Omega \to \R$ be the solution of Cauchy-Neumann problem for the heat equation
\begin{equation}\label{eq:CN}
\begin{cases}
    \partial_t\rho - \Delta \rho=0 & \text{in } (0,+\infty)\times\Omega \\
    \rho(0,\cdot)=\rho_0 & \text{in } \Omega \\
    \nabla \rho \cdot \nn = 0 & \text{on } (0,\infty)\times \partial \Omega.
\end{cases}
\end{equation}
We use the notation $\rho_t:=\rho(t,\cdot)$ and $S_t(\mu_0):=\rho_t\gamma_\Omega$.

Defining the convex function $U:(a,b)\to\R$ by
\begin{equation}\label{eq:DU}
    U''(r)=\frac{1}{h(r)}, \qquad U'((a+b)/2)=0,\quad U((a+b)/2)=0
\end{equation}
and the entropy functional
$$\UU(\rho)=\int_\Omega U(\rho(x))\,dx,$$
we have the following entropy dissipation inequality
\begin{equation}\label{eq:EI}
    \UU(\rho_T)-\UU(\rho_0) \leq - \int_0^T \int_\Omega\frac{|\nabla \rho_s|^2}{h(\rho_s)}\,dx\,ds.
\end{equation}
The inequality \eqref{eq:EI} can be obtained, in the case of smooth initial datum, with a simple computation  and,
in the general case, by a convolution approximation argument.

By Lemma \ref{lemma:dhe}, observing that in our case $\rho_\infty=\frac{c}{\Leb{d}(\Omega)}$, we can prove that there exists $T>0$, independent on $\mu_0$, such that
\begin{equation}\label{boundrho}
    \rho_t \leq \rho_\infty + \frac{b-\rho_\infty}{2}, \qquad \forall t\geq T.
\end{equation}
By the triangular inequality we have that
\begin{equation}\label{cc}
     \Wmod_{\phi,\gamma_\Omega}(\mu_0,\mu_\infty) \leq  \Wmod_{\phi,\gamma_\Omega}(\mu_0,S_T(\mu_0))
     +  \Wmod_{\phi,\gamma_\Omega}(S_T(\mu_0),\mu_\infty).
\end{equation}
Since $h$ is concave and $\Omega$ is bounded,
it is not difficult to see that $\UU$ is bounded in $\MM_{(a,b),c}(\Omega)$,
and recalling \eqref{philphih} we have
$\Wmod_{\phi,\gamma_\Omega}(\mu_0,S_T(\mu_0)) \leq \int_0^T \int_\Omega\frac{|\nabla \rho_s|^2}{h(\rho_s)}\,dx\,ds$,
consequently \eqref{eq:EI} implies that
\begin{equation}\label{ccc}
\sup_{\mu_0\in \MM_{(a,b),c}(\Omega)}\Wmod_{\phi,\gamma_\Omega}(\mu_0,S_T(\mu_0))<+\infty.
\end{equation}
Since
\begin{equation}\label{shift}
\Wmod_{\phi,\gamma_\Omega}(\mu,\nu)=\Wmod_{\tilde\phi,\gamma_\Omega}(\mu-a\gamma_\Omega,\nu-a\gamma_\Omega),
\qquad \text{where } \tilde\phi(r,\ww):=\phi(r+a,\ww),
\end{equation}
considering the new densities $\tilde\rho:=\rho-a$,
and using \eqref{boundrho},
by Proposition \ref{prop:comparison} we obtain
\begin{equation}\label{c}
\sup_{\mu_0\in \MM_{(a,b),c}(\Omega)}\Wmod_{\phi,\gamma_\Omega}(S_T(\mu_0),\mu_\infty)
\leq C \sup_{\mu_0\in \MM_{(a,b),c}(\Omega)}W_2(S_T(\mu_0)-a\gamma_\Omega,\mu_\infty-a\gamma_\Omega)<+\infty,
\end{equation}
because of the boundedness of the Wasserstein distance on the set of measures defined on the bounded convex set $\Omega$.
Finally, we conclude by \eqref{cc}, \eqref{ccc} and \eqref{c}.
\end{proof}

Also in this case, following the proof of Theorem \ref{th:equivconv},
and using the equality \eqref{shift}, Proposition \ref{prop:comparison} and Theorem \ref{th:finOmega},
we can prove the following Theorem.
\begin{theorem}\label{th:equivconvOmega}
Let $\phi\in\AdmissibleDensitytwo (a,b)$
and $\gamma_\Omega:=\chi_{\Omega}\Leb{d}$ with  $\Omega\subset \R^d$ a bounded convex smooth domain.
If $(\mu_n)_{n\in\N}$ is a sequence in $\MM_{(a,b),c}(\Omega)$
weakly-$*$ convergent to $\mu\in\MM_{(a,b),c}(\Omega)$, then
$$\lim_{n\to+\infty}\Wmod_{\phi,\gamma_\Omega}(\mu_n,\mu)=0.$$
\end{theorem}

We recall that the space of non-negative measures with fixed mass $c>0$, supported on a bounded convex open set,
is bounded with respect to the standard Wasserstein distance (easy consequence of the definition),
and the convergence with respect to the standard Wasserstein distance is equivalent to the weak$^*$ one on this set.
Theorems \ref{th:finOmega} and \ref{th:equivconvOmega} state that the analogous properties hold
for the space $\MM_{(a,b),c}(\Omega)$ endowed with the distance $\Wmod_{\phi,\gamma_\Omega}$.

\subsubsection{Appendix: decay for heat equation}
In this appendix we recall a standard result on the asymptotic behavior of the heat equation.
Since it seems not simple to find it in this form, we also give a proof.
\begin{lemma}\label{lemma:dhe}
Let $\Omega$ be a convex smooth domain of $\R^d$.
If $\rho_0:\Omega \to [a,b]$, and $\rho:(0,+\infty)\times\Omega\to \R$ denotes the solution of the problem
\begin{equation}\label{eq:CN2}
\begin{cases}
    \partial_t\rho - \Delta \rho=0 & \text{in } (0,+\infty)\times\Omega \\
    \rho(0,\cdot)=\rho_0 & \text{in } \Omega \\
    \nabla \rho \cdot \nn = 0 & \text{on } (0,\infty)\times \partial \Omega ,
\end{cases}
\end{equation}
then there exist two constants
$C>$ and $\lambda>0$, depending only on $a$, $b$ and $\Omega$ such that
\begin{equation}\label{expdecay}
    || \rho_s - \rho_\infty ||_{L^\infty(\Omega)} \leq C e^{-\lambda s}, \qquad \forall\,s\geq 0 ,
\end{equation}
where $\rho_s:=\rho(s,\cdot)$ and $\rho_\infty:=\frac{1}{\Leb{d}(\Omega)}\int_\Omega \rho_0(x)\,dx$.
\end{lemma}

\begin{proof}
Since  $\partial_t(\rho_t-\rho_\infty) - \Delta (\rho_t-\rho_\infty)=0$
with homogeneous Neumann boundary conditions,
multiplying this equation by $\rho_t-\rho_\infty$ and integrating by parts we obtain the identity
\begin{equation}\label{eq:EI2}
    \frac{d}{dt}||\rho_t-\rho_\infty||_{L^2(\Omega)}^2 + 2||\nabla \rho_t||_{L^2(\Omega)}^2=0.
\end{equation}
By Poincar\'e's inequality, there exists a constant $C_P$ depending only on $\Omega$ such that
\begin{equation}
    ||\nabla \rho_t||_{L^2(\Omega)}^2 \geq C_P ||\rho_t-\rho_\infty||_{L^2(\Omega)}^2,
\end{equation}
and from \eqref{eq:EI2} we immediately obtain the $L^2(\Omega)$ exponential decay
\begin{equation}\label{L2expdecay}
    || \rho_t - \rho_\infty ||_{L^2(\Omega)} \leq e^{-C_P t} || \rho_0 - \rho_\infty ||_{L^2(\Omega)},
    \qquad \forall\,t\geq 0 .
\end{equation}
The $L^2(\Omega)-W^{1,\infty}(\Omega)$ interpolation inequality
(see for instance \cite[Complements of Chapter IX]{Brezis} or \cite{Nirenberg}),
states that there exist a constant $C$ depending only on $\Omega$ such that
\begin{equation}\label{interp}
    || \rho_t - \rho_\infty ||_{L^\infty(\Omega)} \leq
    C || \rho_t - \rho_\infty ||_{L^2(\Omega)}^{2/(d+2)} || \rho_t - \rho_\infty ||_{W^{1,\infty}(\Omega)}^{d/(d+2)}
    \qquad \forall\,t\geq 0 .
\end{equation}
In order to get a uniform bound of the $L^\infty$ norm of the gradient, we define
$v(t,x):=\rho^2_t(x)+t|\nabla \rho_t(x)|^2$, which solves the problem
\begin{equation}\label{ineqproblem}
\begin{cases}
    \partial_t v - \Delta v \leq 0 & \text{in } (0,+\infty)\times\Omega \\
    v(0,\cdot)=\rho_0^2 & \text{in } \Omega \\
    \nabla v \cdot \nn \leq 0 & \text{on } (0,\infty)\times \partial \Omega .
\end{cases}
\end{equation}
Indeed, by a simple computation we have that $v$ satisfies the partial
differential inequality in \eqref{ineqproblem}.
In order to obtain the boundary
condition satisfied by $v$ we have $\nabla v \cdot \nn = \nabla \rho^2 \cdot \nn + t
\nabla |\nabla \rho|^2 \cdot \nn = t \nabla |\nabla \rho|^2 \cdot \nn$ because
of the boundary condition in \eqref{eq:CN2}. Moreover, by the smoothness and
the convexity of $\Omega$, we have that $\nabla |\nabla \rho|^2 \cdot \nn \leq
0$ (see for instance \cite[Lemma 5.1]{GST}).

The maximum principle for problem \eqref{ineqproblem} (see for instance \cite{Friedman}) states that
$v(t,x) \leq  ||\rho_0^2||_{L^\infty(\Omega)}$.
In particular we have
\begin{equation}\label{gradientdecay}
 \sqrt{t} ||\nabla \rho_t||_{L^\infty(\Omega)} \leq ||\rho_0||_{L^\infty(\Omega)} \leq \max (|a|,|b|).
\end{equation}
The inequality \eqref{expdecay} follows from \eqref{interp} and \eqref{gradientdecay} (for $t\geq 1$ for instance) and \eqref{L2expdecay},
recalling that $||\rho_t-\rho_\infty||_{L^\infty(\Omega)} \leq 2 \max (|a|,|b|)$.
\end{proof}

\bigskip

\noindent {\bf Acknowledgements:} The authors would like to thank Giuseppe Savar\'e 
for useful suggestions on this work.

%


\begin{thebibliography}{99}
\bibitem{AB}
L. Ambrosio, G. Buttazzo, \emph{Weak lower semicontinuous envelope
of functionals defined on a space of measures},
Ann. Mat. Pura Appl., 150 (1988), pp.311--339.

\bibitem{AGS}
L. Ambrosio, N. Gigli, G. Savar\'e, {Gradient flows in Metric
Spaces and in the Space of Probability Measures},
Birk\"auser Verlag, Basel 2005

\bibitem{AFP}
L. Ambrosio, N. Fusco, D. Pallara, {Functions of bounded variation
and free discontinuity problems},
Oxford Mathematical Monographs, Claredon Press, Oxford, 2000.

\bibitem{BB}
J.-D. Beneamou, Y. Brenier,\emph{A computational fluid mechanics solution to the
Monge-Kantorovich mass transfer problem},
Numer. Math., 84 (2000), pp.375--393.

\bibitem{Ber98}
A.~Bertozzi, \emph{The mathematics of moving contact lines in thin
liquid films}, Notices Amer. Math. Soc., 45 (1998), pp.~689-697.

\bibitem{Brezis}
H. Brezis. Analyse fonctionnelle. Masson,
Paris, 1983.

\bibitem{BFD06}
M. Burger, M. di Francesco and Y. Dolak, \emph{The Keller-Segel model
for chemotaxis with prevention of overcrowding: linear vs.
nonlinear diffusion}, SIAM J. Math. Anal., 38 (2006),
pp.~1288--1315.

\bibitem{BDi}
M. Burger and M. Di Francesco, \emph{Large time behavior of nonlocal
aggregation models with nonlinear diffusion},  Netw. Heterog. Media, 3 (2008), pp. 749--785.

\bibitem{CLR08} J. A. Carrillo, P. Lauren\c{c}ot and J. Rosado,
\emph{Fermi-Dirac-Fokker-Planck Equation: Well-posedness \& Long-time
Asymptotics}, J. Differential Equations, 247 (2009), pp. 2209--2234.

\bibitem{CLSS} J.A. Carrillo, S. Lisini, G. Savar\'e, D. Slep\v{c}ev,
\emph{Nonlinear mobility continuity equations and generalized displacement convexity}
 arXiv:0901.3978v1 [math.AP] to appear on J. Funct. Anal.

\bibitem{CRS08} J. A. Carrillo, J. Rosado and F. Salvarani, \emph{1D nonlinear
Fokker-Planck equations for fermions and bosons}, Appl. Math.
Lett., 21 (2008), pp.~148-154.

\bibitem{DiR}
M. Di Francesco and J. Rosado, \emph{Fully parabolic Keller-Segel model
for chemotaxis with prevention of overcrowding}, Nonlinearity, 21 (2008), pp. 2715--2730.

\bibitem{DNS} J. Dolbeault, B. Nazaret, G. Savar\'e, \emph{A new class of transport distances between measures},
Calc. Var. Partial Differential Equations, 34 (2009), pp. 193--231.

\bibitem{EG} C. M. Elliott, H. Garcke, \emph{On the Cahn-Hilliard equation with degenerate mobility},
SIAM J. Math. Anal., 27 (1996), pp. 404--423.

\bibitem{bib:f1}
T. D. Frank, \emph{Classical Langevin equations for the free electron
gas and blackbody radiation}, J. Phys. A, 37 (2004),
pp.~3561--3567.

\bibitem{bib:f2}
T. D. Frank, Nonlinear Fokker-Planck Equations, Springer Series in
Synergetics, Springer, 2005.

\bibitem{Friedman}
A. Friedman, Partial differential equations of parabolic type.
Prentice-Hall, 1964.

\bibitem{GL1}
G.~ Giacomin and J.~ Lebowitz, \emph{Phase segregation dynamics in
particle systems with long range interactions. {I}. {M}acroscopic
limits}, J. Statist. Phys., 87 (1997), pp.~37--61.

\bibitem{GST}
U.~Gianazza, G.~Savar{\'e} and G.~Toscani, \emph{The {W}asserstein
gradient flow of the {F}isher information and the quantum
drift-diffusion equation}, Arch. Ration. Mech. Anal., 194 (2009), pp. 133--220.

\bibitem{bib:k}
G. Kaniadakis, \emph{Generalized Boltzmann equation describing the
dynamics of bosons and fermions}, Phys. Lett. A, 203 (1995),
pp.~229--234.

\bibitem{bib:k3}
G. Kaniadakis, P. Quarati, \emph{Kinetic equation for classical
particles obeying an exclusion principle}, Phys. Rev. E, 48 (1993),
pp.~4263--4270.

\bibitem{LMS}
S. Lisini, D. Matthes, G. Savar\'e, (paper in preparation) (2009).


\bibitem{Nirenberg}
L. Nirenberg, \emph{On elliptic partial differential equations}, Ann. Scuola Norm. Sup. Pisa, 3 (1959), pp. 115--162.

\bibitem{R} R. T. Rockafellar, {Convex Analysis}, Princeton University, Princeton, 1970.


\bibitem{Sl08}
D. Slep\v{c}ev, \emph{Coarsening in nonlocal interfacial systems}, SIAM
J. Math. Anal., 40 (2008), pp.~1029--1048.

\bibitem{Vil03} C. Villani, Topics in optimal transportation. Graduate Studies in Mathematics, 58.
American Mathematical Society, Providence, RI, 2003.

\bibitem{Vil09} C. Villani, Optimal transport. Old and new.
Grundlehren der Mathematischen Wissenschaften, 338. Springer-Verlag, Berlin, 2009.

\end{thebibliography}
\end{document}